\documentclass[12pt]{amsart}
\usepackage{amsmath,amsthm,amsfonts,amscd,amssymb,eucal,latexsym,mathrsfs, appendix, yhmath, fancyhdr}
\usepackage[numbers,sort&compress]{natbib}
\usepackage[all,cmtip]{xy}

\newtheorem{theorem}{Theorem}[section]
\newtheorem{corollary}[theorem]{Corollary}
\newtheorem{lemma}[theorem]{Lemma}
\newtheorem{proposition}[theorem]{Proposition}

\theoremstyle{definition}
\newtheorem{definition}[theorem]{Definition}
\newtheorem{remark}[theorem]{Remark}
\newtheorem{notation}[theorem]{Notation}
\newtheorem{example}[theorem]{Example}
\newtheorem{question}[theorem]{Question}

\def\Map{{\rm Map}}
\def\BAD{{\rm BAD}}
\def\GOOD{{\mathcal{G}}}
\def\Vert{{\rm Vert}}
\def\N{{\mathbb N}}
\def\cc{{\curvearrowright}}
\def\R{{\mathbb R}}
\def\Z{{\mathbb Z}}

\newcommand{\cA}{{\mathcal A}}
\newcommand{\cB}{{\mathcal B}}
\newcommand{\cG}{{\mathcal G}}
\newcommand{\cH}{{\mathcal H}}

\newcommand{\cC}{{\mathcal C}}

\newcommand{\cN}{{\mathcal N}}

\newcommand{\cT}{{\mathcal T}}

\newcommand{\cF}{{\mathcal F}}

\newcommand{\Cb}{{\mathbb C}}

\newcommand{\Zb}{{\mathbb Z}}

\newcommand{\Rb}{{\mathbb R}}
\newcommand{\Nb}{{\mathbb N}}

\newcommand{\tr}{{\rm tr}}

\newcommand{\ddet}{{\rm det}}
\newcommand{\Fix}{{\rm Fix}}
\newcommand{\fe}{{\mathfrak e}}

\newcommand{\T}{{\mathbb T}}
\newcommand{\C}{{\mathbb C}}

\newcommand{\vs}{{\vec{s}}}
\newcommand{\Ends}{{\rm{Ends}}}

\allowdisplaybreaks

\begin{document}

\title{Harmonic models and spanning forests of residually finite groups}
\author{Lewis Bowen}
\author{Hanfeng Li}
\address{\hskip-\parindent
Lewis Bowen, Department of Mathematics, Texas A{\&}M University,
College Station, TX 77843-3368, U.S.A.}
\email{lpbowen@math.tamu.edu}

\address{\hskip-\parindent
Hanfeng Li, Department of Mathematics, SUNY at Buffalo,
Buffalo, NY 14260-2900, U.S.A.}
\email{hfli@math.buffalo.edu}

\keywords{Harmonic model, algebraic dynamical system, Wired Spanning Forest, WSF, tree entropy}

\begin{abstract}
We prove a number of identities relating the sofic entropy of a certain class of non-expansive algebraic dynamical systems, the sofic entropy of the Wired Spanning Forest and the tree entropy of Cayley graphs of residually finite groups. We also show that homoclinic points and periodic points in harmonic models are dense under general conditions.
\end{abstract}
\date{August 13, 2011}

\maketitle

\tableofcontents

\section{Introduction}

This paper is concerned with two related dynamical systems. The quickest way to explain the connections is to start with finite graphs. So consider a finite connected simple graph $G=(V,E)$. The graph Laplacian $\Delta_G$ is an operator on $\ell^2(V,\R)$ given by $\Delta_G x(v) = \sum_{w: \{v,w\}\in E} (x(v)-x(w))$. Let $\det^*(\Delta_G)$ be the product of the non-zero eigenvalues of $\Delta_G$.  By the Matrix-Tree Theorem (see, e.g., \cite{GR} Lemma 13.2.4), $|V|^{-1}\det^*(\Delta_G)$, is the number of spanning trees in $G$. Recall that a subgraph $H$ of $G$ is {\em spanning} if it contains every vertex. It is a {\em forest} if it has no cycles. A connected forest is a {\em tree}. The number of spanning trees in $G$ is denoted $\tau(G)$.

There is another interpretation for this determinant. Consider the space $(\R/\Z)^V$ of all functions $x:V \to \R/\Z$. The operator $\Delta_G$ acts on this space as well by the same formula. An element $x\in (\R/\Z)^V$ is {\em harmonic mod $1$} if $\Delta_G x \in \Z^V$. The set of harmonic mod $1$ elements is an additive group $X_G < (\R/\Z)^V$ containing the constants. The number of connected components of this group is denoted $|X_G|$. Let $\Z^V_0$ be the set of integer-valued functions $x:V \to \Z$ with zero sum: $\sum_{v\in V} x(v)=0$. Because $\Delta_G$ maps $\Z_0^V$ injectively into itself, $\det^*(\Delta_G) = |V||X_G|$. To our knowledge, this was first observed in \cite{So98}. We provide more details in \S \ref{S-approximation}.

Our main results generalize the equalities $\det^*(\Delta_G)|V|^{-1} = \tau(G) = |X_G|$ to Cayley graphs of finitely generated residually finite groups.

\subsection{Harmonic models and other algebraic systems}
We begin with a discussion of the appropriate analogue of the group of harmonic mod $1$ points and, to provide further context, the more general setting of group actions by automorphisms of a compact group. This classical subject has long been studied when the acting group is $\Z$ or $\Z^d$ \cite{Schmidt} though not as much is known in the general case.

Let $\Gamma$ be a countable group and $f$ be an element in the integral group ring $\Z\Gamma$. The action of $\Gamma$ on the discrete abelian group $\Z\Gamma/\Z\Gamma f$ induces an action by automorphisms on the Pontryagin dual $X_f:=\widehat{\Z\Gamma/\Z\Gamma f}$, the compact abelian group of all homomorphisms from $\Z\Gamma/\Z\Gamma f$ to $\T$, the unit circle in $\C$. Call the latter action $\alpha_f$. The topological entropy and measure-theoretic entropy (with respect to the Haar probability measure) coincide when $\Gamma$ is amenable \cite{Berg, De06}. Denote this number by $h(\alpha_f)$.

\begin{example} If $S \subset \Gamma \setminus \{e\}$ is a finite symmetric generating set, where $e$ denotes the identity element of $\Gamma$, and $f$ is defined by $f_s=-1$ if $s\in S$, $f_e=|S|$ and $f_s=0$ for $s \notin S \cup \{e\}$ then $X_f$ is canonically identified with the group $\{x \in (\R/\Z)^\Gamma:~ \sum_{s\in S} x_{ts} = |S|x_t, \forall t\in \Gamma\}$ of harmonic mod $1$ points of the associated Cayley graph.
\end{example}

Yuzvinskii proved \cite{Yu65, Yu67} that if $\Gamma=\Z$ then the entropy of $\alpha_f$ is calculable as follows. When $f=0$, $h(\alpha_f)=+\infty$. When $f\neq 0$, write $f$ as a Laurent polynomial $f = u^{-k}\sum_{j=0}^n c_j u^j$ by identifying $1 \in \Z=\Gamma$ with the indeterminate $u$ and requiring that $c_nc_0\ne 0, n\ge 0$. If $\lambda_1,\ldots, \lambda_n$ are the roots of $\sum_{j=0}^n c_j u^j$, then
$$h(\alpha_f) = \log |c_n|  + \sum_{j=1}^n \log^+ |\lambda_j|$$
where $\log^+ t = \log \max(1,t)$. More generally, Yuzvinskii developed formulas for the entropy for any endomorphism of a compact metrizable group \cite{Yu67}.

When $\Gamma=\Z^d$, we identify $\Z\Gamma$ with the Laurent polynomial ring $\Z[u_1^{\pm 1},\ldots, u_d^{\pm1}]$. Given a nonzero Laurent polynomial $f \in \Z\Gamma$, the Mahler measure of $f$ is defined by
$$\mathbb{M}(f)=\exp \left(\int_{\T^d} |f(s)|~ds\right)$$
where the integral is with respect to the Haar probability measure on the torus $\T^d$. In \cite{LSW90} it is shown that $h(\alpha_f) = \log \mathbb{M}(f)$. This is a key part of a more general procedure for computing the entropy of any action of $\Z^d$ by automorphisms of a compact metrizable group.

In \cite{FK}, Fuglede and Kadison introduced a determinant $\det_A f$ for elements $f$ of a von Neumann algebra $A$ with respect to a normal tracial state $\tr_A$. It has found widespread application in the study of $L^2$-invariants \cite{Luck}. We will apply it to the special case when $A$ is the group von Neumann algebra $\cN\Gamma$ with respect to its natural trace $\tr_{\cN\Gamma}$. Note that  $\Z\Gamma$ is a sub-ring of $\cN\Gamma$. These concepts are reviewed in \S \ref{S-notation}.


As explained in \cite[Example 3.13]{Luck}, if $\Gamma = \Z^d$ and $f \in \Z\Gamma$ is nonzero, then the Mahler measure of $f$ equals the logarithm of its Fuglede-Kadison determinant. So it was natural to wonder whether the equation $h(\alpha_f) = \log \det_{\cN\Gamma} f$ holds whenever $\Gamma$ is amenable and $f$ is invertible in $\cN\Gamma$. Some special cases were proven in \cite{De06} and \cite{DS} before the general case was solved in the affirmative in \cite{Li}.

Recall that $\Gamma$ is {\em residually finite} if the intersection of all finite-index subgroups of $\Gamma$ is the trivial subgroup. In this case, there exists a sequence of finite-index normal subgroups $\Gamma_n\lhd \Gamma$ such that $\bigcap_{n=1}^\infty \bigcup_{i\ge n} \Gamma_i = \{e\}$. Our main results concern residually finite groups; we do not, in general, require that $\Gamma$ is amenable.

A group $\Gamma$ is {\em sofic} if it admits a sequence of partial actions on finite sets which, asymptotically, are free actions. This large class of groups, introduced implicitly in \cite{Gr99} and developed further in \cite{We00, ES05, ES06}, contains all residually finite groups and all amenable groups. It is not known whether all countable groups are sofic.

Entropy was introduced in \cite{Ko58,Ko59} for actions of $\Z$. The definition and major results were later extended to all countable amenable groups \cite{Ki75,Ol85, OW87}. Until recently it appeared to many observers that entropy theory could not be extended beyond amenable groups. This changed when \cite{Bo10a} introduced an entropy invariant for actions of free groups. Soon afterwards, \cite{Bo10} introduced entropy for probability-measure-preserving sofic group actions. One disadvantage of the approach taken in \cite{Bo10} (and in \cite{Bo10a}) is that it only applies to actions with a finite-entropy generating partition. This requirement was removed in \cite{KL11a}. That paper also introduced topological sofic entropy for actions of sofic groups on compact metrizable spaces and proved a variational principle relating the two concepts analogous to the classical variational principle.

If $\Gamma$ is non-amenable then the definition of entropy of a $\Gamma$-action depends on a choice of sofic approximation. We will not need the full details here because we are only concerned with the special case in which $\Gamma$ is a residually finite group. A sequence $\Sigma=\{\Gamma_n\}^\infty_{n=1}$ of finite-index normal subgroups of $\Gamma$ satisfying $\bigcap_{n=1}^\infty \bigcup_{i\ge n} \Gamma_i = \{e\}$ determines, in a canonical manner, a sofic approximation to $\Gamma$. Thus, we let $h_{\Sigma,\lambda}(\alpha_f)$ and $h_\Sigma(\alpha_f)$ denote the measure-theoretic sofic entropy and the topological sofic entropy of $\alpha_f$ with respect to $\Sigma$ respectively. Here $\lambda$ denotes the Haar probability measure on $X_f$. The precise definition of sofic entropy is given in \S \ref{sec:sofic} below.

In \cite{Bo11}, it was proven that if $f \in \Z\Gamma$ is invertible in $\ell^1(\Gamma)$ then $h_{\Sigma,\lambda}(\alpha_f) = \log \det_{\cN\Gamma} f$ as expected. Also, if $f$ is invertible in the universal group $C^*$-algebra of $\Gamma$ then by \cite{KL11a}, $h_{\Sigma}(\alpha_f) = \log \det_{\cN\Gamma} f$.

\begin{definition}\label{defn:well-balanced}
We say that $f \in \R\Gamma$ is {\em well-balanced} if
\begin{enumerate}
\item $\sum_{s\in \Gamma}f_s=0$,

\item $f_s\le 0$ for every $s\in \Gamma\setminus \{e\}$,

\item $f=f^*$ (where $f^*$, the adjoint of $f$ is given by $f^*_s = f_{s^{-1}}$ for all $s\in \Gamma$).

\item the support of $f$ generates $\Gamma$.
\end{enumerate}
\end{definition}
If $f \in \Z\Gamma$ is well-balanced then the dynamical system $\Gamma \cc X_f$ is called a {\em harmonic model} because $X_f$ is naturally identified with the set of all $x \in (\R/\Z)^\Gamma$ such that $xf =0$, i.e., $x$ satisfies the harmonicity equation mod $1$:
$$\sum_{s\in \Gamma} x_{ts}f_s = 0 ~\mod \Z$$
for all $t\in \Gamma$.

If $\Gamma=\Z^d$ ($d\ge 2$), then much is known about the harmonic model: the entropy was computed in \cite{LSW90} in terms of Mahler measure, it follows from \cite[Theorem 7.2]{KS} that the periodic points are dense, \cite{SV} provides an explicit description of the homoclinic group in some cases, and \cite{LSV} contains a number of results on the homoclinic group, periodic points, the specification property and more in some cases. See also \cite{SV} for results relating the harmonic model to the abelian sandpile model.

Our first main result is:
\begin{theorem}\label{thm:main1}
Let $\Gamma$ be a countably infinite group and $\Sigma=\{\Gamma_n\}^\infty_{n=1}$ a sequence of finite-index normal subgroups of $\Gamma$ satisfying $\bigcap_{n=1}^\infty \bigcup_{i\ge n} \Gamma_i = \{e\}$. Let $f \in \Z\Gamma$ be well-balanced. Then
$$h_{\Sigma,\lambda}(\alpha_f) = h_\Sigma(\alpha_f) = \log \ddet_{\cN\Gamma} f = \lim_{n\to\infty} [\Gamma:\Gamma_n]^{-1} \log |\Fix_{\Gamma_n}(X_f)|$$
where $|\Fix_{\Gamma_n}(X_f)|$ is the number of connected components of the set of fixed points of $\Gamma_n$ in $X_f$.
\end{theorem}

In the appendix, we show that if $f$ is well-balanced then it is not invertible in $\ell^1(\Gamma)$ or even in the universal group $C^*$-algebra of $\Gamma$. Moreover, it is invertible in $\cN\Gamma$ if and only if $\Gamma$ is non-amenable. Thus Theorem \ref{thm:main1} does not have much overlap with previous results.

In order to prove this theorem, we obtain several more results of independent interest. 
We show that  if $\Gamma$ is not virtually isomorphic to $\Z$ or $\Z^2$ then the homoclinic subgroup of $X_f$ is dense in $X_f$. 
This implies $\Gamma \cc X_f$ is mixing of all orders (with respect to the Haar probability measure). 
Also, $\Fix_{\Gamma_n}(X_f)$ converges to $X_f$ in the Hausdorff topology as $n\to\infty$.  
As far we know, these results were unknown except in the case $\Gamma=\Z^d$.   
Indeed, in the case $\Gamma=\Z^d$,  it follows from \cite[Theorem 7.2]{KS} that the periodic points are dense and 
it is known that the action is mixing of all orders if $d \ge 2$ (see Remark \ref{remark:homoclinic} for details). 

\subsection{Uniform Spanning Forests}
In this paper, all graphs are allowed to have multiple edges and loops. A subgraph of a graph  is  {\em spanning} if it contains every vertex. It is a {\em forest} if every connected component is simple connected (i.e., has no cycles). A connected forest is a {\em tree}. If $G$ is a finite connected graph then the {\em Uniform Spanning Tree} (UST) on $G$ is the random subgraph whose law is uniformly distributed on the collection of spanning trees. Motivated in part to develop an analogue of the UST for infinite graphs, R. Pemantle implicitly introduced in \cite{Pe91} the {\em Wired Spanning Forest} (WSF) on $\Z^d$. This model has been generalized to arbitrary locally finite graphs and studied intensively (see e.g., \cite{BLPS01}).

In order to define the WSF, let $G=(V,E)$ be a locally finite connected graph and $\{G_n\}_{n=1}^\infty$ an increasing sequence of finite subgraphs $G_n=(V_n,E_n) \subset G$ whose union is all of $G$. For each $n$, define the wired graph $G^w_n=(V^w_n,E^w_n)$ as follows. The vertex set $V^w_n= V_n \cup \{*\}$. The edge set $E_n^w$ of $G_n^w$ contains all edges in $G_n$. Also for every edge $e$ in $G$ with one endpoint $v$ in $G_n$ and the other endpoint $w$ not in $G_n$, let $e^*=\{v,*\}$. Then $G^w_n$ contains all edges of the form $e^*$. These are all of the edges of $G^w_n$. Let $\nu^w_n$ be the uniform probability measure on the set of spanning trees of $G^w_n$. We consider it as a probability measure on the set $2^{E_n^w}$ of all subsets of $E_n^w$. Because $E_n \subset E$, we can think of $2^{E_n}$ (which we identify with the set of all subsets of $E_n$) as a subset of $2^E$  (which we identify with the set of all subsets of $E$). The projection of $\nu^w_n$ to $2^{E_n} \subset 2^E$ converges (as $n\to\infty$) to a Borel probability measure $\nu_{WSF}$ on $2^E$ (in the weak* topology on the space of Borel probability measures of $2^E$). This measure does not depend on the choice of $\{G_n\}_{n=1}^\infty$. This was implicitly proven by Pemantle \cite{Pe91} (answering a question of R. Lyons). By definition, $\nu_{WSF}$ is the law of the Wired Spanning Forest on $G$. There is a related model, called the Free Spanning Forest (FSF) (obtained by using $G_n$ itself in place of $G^w_n$ above) which is frequently discussed in comparison with the WSF. However, we will not make any use of it here. For more background, the reader is referred to \cite{BLPS01}.

\begin{definition}\label{defn:Cayley}
Let $\Gamma$ be a countably infinite group and let $f \in \Z\Gamma$ be well-balanced (Definition~\ref{defn:well-balanced}). The Cayley graph $C(\Gamma,f)$ has vertex set  $\Gamma$. For each $v \in \Gamma$ and $s \ne e$, there are $|f_s|$  edges between $v$ and $vs$. Let $E(\Gamma,f)$ denote the set of edges of $C(\Gamma,f)$. Similarly, for $\Gamma_n \triangleleft \Gamma$, we let $C_n^f=C(\Gamma/\Gamma_n,f)$ be the graph with vertex set $\Gamma/\Gamma_n$ such that for each $g\Gamma_n \in \Gamma/\Gamma_n$ and $s \in \Gamma$ there are $|f_s|$ edges between $g\Gamma_n$ and $gs\Gamma_n$. We let $E_n^f$ be the set of edges of $C^f_n$.
\end{definition}


Let $\nu_{WSF}$ be the law of the WSF on $C(\Gamma,f)$.  It is a probability measure on $2^{E(\Gamma,f)}$. Of course, $\Gamma$ acts on $E(\Gamma,f)$ which induces an action on $2^{E(\Gamma,f)}$ which preserves $\nu_{WSF}$.

In \cite{Lyons05}, R. Lyons  introduced the {\em tree entropy} of a transitive weighted graph (and more generally, of a probability measure on weighted rooted graphs). In our case, the definition runs as follows. Let $\mu$ be the probability measure on $\Gamma$ defined by $\mu(s)=|f_s|/f_e$ for $s\in \Gamma \setminus \{e\}$. Then the tree entropy of $C(\Gamma,f)$ is
$${\bf h}(C(\Gamma, f)):=  \log f_e - \sum_{k\ge 1} k^{-1} \mu^k(e)$$
where $\mu^k$ denotes the $k$-fold convolution power of $\mu$. In probabilistic terms, $\mu^k(e)$ is the probability that a random walk with i.i.d. increments with law $\mu$ started from $e$ will return to $e$ at time $k$. In \cite{Lyons05}, R. Lyons proved that if $\Gamma$ is amenable then the measure-entropy of $\Gamma \cc (2^{E(\Gamma,f)}, \nu_{WSF})$ equals ${\bf h}(C(\Gamma,f))$. We extend this result to all finitely-generated residually finite groups (where entropy is taken with respect to a sequence $\Sigma$ of finite-index normal subgroups converging to the trivial subgroup).


In \cite[Theorem 3.1]{Lyons10}, it is shown that ${\bf h}(C(\Gamma,f)) = \log \det_{\cN\Gamma} f$. We give another proof in \S \ref{S-approximation}. Thus by Theorem \ref{thm:main1}, the entropy of the WSF equals the entropy of the associated harmonic model.

In the case $\Gamma=\Z^d$, the entropy of the harmonic model was computed in \cite{LSW90}. Then the topological entropy of the action of $\Z^d$ on the space of essential spanning forests was computed in \cite{BP93} and shown to coincide with the entropy of the harmonic model. This coincidence was mysterious until \cite{So98} explained how to derive this result without computing the entropy. 

We also study a topological model related to the WSF (though not the same model as in \cite{BP93}). To describe it, we need to introduce some notation.

\begin{notation}\label{note:S}
Let $S$ be the set of all $s\in \Gamma \setminus\{e\}$ with $f_s \ne 0$. Let $S_* \subset E(\Gamma,f)$ denote the set of edges adjacent to the identity element. Let $p:S_* \to S$ be the map which takes an edge with endpoints $\{e,g\}$ to $g$. Also, let $s_*^{-1}=s^{-1}s_* \in S_*$ where $p(s_*)=s$. Note that $p(s_*^{-1})= p(s_*)^{-1}$ and $(s_*^{-1})^{-1}=s_*$. 
\end{notation}


 An element $y\in S_*^\Gamma$ defines, for each $g\in \Gamma$ a directed edge in $C(\Gamma,f)$ away from $g$ (namely, the edge $gs$ where $y_g=s$).
Let $\cF_f \subset S_*^\Gamma$ be the set of all elements $y \in S_*^\Gamma$ such that
\begin{enumerate}
\item edges are oriented consistently: if $y_g=s_*$ and $p(s_*)=s$ then $y_{gs} \ne s_*^{-1}$.
\item there are no cycles: there does not exist $g_1,g_2,\ldots, g_n\in \Gamma$ and $s_1,\ldots, s_n \in S$ such that $g_is_i=g_{i+1}$, $g_1=g_n$ and $p(y_{g_i})=s_i$ for $1\le i \le n-1$.
\end{enumerate}
The space $\cF_f \subset S_*^\Gamma$ is closed (where $S_*^\Gamma$ is given the product topology) and is therefore a compact metrizable space. Let $h_\Sigma(\cF_f,\Gamma)$ denote the topological sofic entropy of the action $\Gamma \cc \cF_f$. Our second main result is:

\begin{theorem}\label{thm:WSF}
If $\Gamma$ is a countably infinite group, $\Sigma=\{\Gamma_n\}_{n=1}^\infty$ is a sequence of finite-index normal subgroups of $\Gamma$, $\bigcap_{n=1}^\infty \bigcup_{i \ge n} \Gamma_i = \{e\}$ and $f \in \Z\Gamma$ is well-balanced then
$$h_{\Sigma,\nu_{WSF}}(2^{E(\Gamma,f)},\Gamma) = h_{\Sigma}(\cF_f,\Gamma) = {\bf h}(C(\Gamma,f))=\log \ddet_{\cN\Gamma} f = \lim_{n\to\infty} [\Gamma:\Gamma_n]^{-1} \log \tau(C_n^f)$$
where $\tau(C_n^f)$ is the number of spanning trees of the Cayley graph $C(\Gamma/\Gamma_n,f)$.
\end{theorem}

{\bf Acknowledgements}: L. B. would like to thank  Russ Lyons for asking the question,  ``what is the entropy of the harmonic model on the free group?'' which started this project. L.B. would also like to thank Andreas Thom for suggesting the use of Property (T) to prove an upper bound on the sofic entropy of the harmonic model. This idea turned out to be a very useful entry way into the problem. Useful conversations occurred while L.B. visited the Banff International Research Station, the Mathematisches Forschungsinstitut Oberwolfach, and the Institut Henri Poincar\'e. L.B. was partially supported by NSF grants DMS-0968762 and DMS-0954606.

 H. L. was partially supported by NSF Grants DMS-0701414 and DMS-1001625. This work was carried out while H.L. visited the Institut Henri Poincar\'{e} and the math departments of Fudan University and University of Science and Technology of China in the summer of 2011. He thanks Wen Huang, Shao Song, Yi-Jun Yao, and Xiangdong Ye for warm hospitality. H.L. is also grateful to Klaus Schmidt for sending him the manuscript \cite{LS}.

\section{Notation and preliminaries} \label{S-notation}

Throughout the paper, $\Gamma$ denotes a
countable group with identity element $e$. We denote $f \in \ell^p(\Gamma):=\ell^p(\Gamma, \Rb)$ by $f=(f_s)_{s\in \Gamma}$. Given $g\in \ell^1(\Gamma)$ and $h\in \ell^\infty(\Gamma)$ we define the convolution $g h \in \ell^\infty(\Gamma)$ by
  \begin{eqnarray*}
  (gh)_w &:=& \sum_{s \in \Gamma} g_s h_{s^{-1}w} =\sum_{s \in \Gamma} g_{ws^{-1}} h_s, \quad \forall w\in \Gamma.
   \end{eqnarray*}
  More generally, whenever $g$ and $h$ are functions of $\Gamma$ we define $gh$ as above whenever this formula is well-defined. Let $\Z\Gamma\subset \ell^\infty(\Gamma)$ be the subring of all elements $f$ with $f_s \in \Z~\forall s \in \Gamma$ and $f_s=0$ for all but finitely many $s\in \Gamma$. Similarly $\R\Gamma\subset \ell^\infty(\Gamma)$ is the subring of all elements $f$ with $f_s \in \R~\forall s \in \Gamma$ and $f_s=0$ for all but finitely many $s\in \Gamma$.  The element $1 \in \Z\Gamma$ is defined by $1_e=1, 1_s =0 ~\forall s\in\Gamma \setminus\{e\}$.

    Given sets $A$ and $B$, $A^B$ denotes the set of all functions from $B$ to $A$. We frequently identify the unit circle $\T$  with $\R/\Z$. If $g\in \Z\Gamma$  and $h\in \T^\Gamma$ then define the convolutions $g h$ and $h g \in \T^\Gamma$ as above.

  The {\em adjoint} of an element $g \in \C^\Gamma$ is defined by $g^*(s) := \overline{g(s^{-1})}$.

  A probability measure $\mu$ on $\Gamma$ can be thought of as an element of $\ell^1(\Gamma)$. Thus we write $\mu_s=\mu(\{s\})$ for any $s\in \Gamma$, and $\mu^n$ denotes the $n$-th convolution power of $\mu$.


 The group $\Gamma$ acts isometrically on the Hilbert space $\ell^2(\Gamma, \Cb)$ from the left by $(sf)_t=f_{s^{-1}t}$ for all $f\in \ell^2(\Gamma, \Cb)$ and  $s, t\in \Gamma$. Denote by $\cB(\ell^2(\Gamma, \Cb))$ the algebra of bounded linear operators from $\ell^2(\Gamma, \Cb)$ to itself.
 The group von Neumann algebra $\cN \Gamma$ is the algebra of elements in $\cB(\ell^2(\Gamma, \Cb))$ commuting with the left action of $\Gamma$ (see \cite[Section 2.5]{BO} for detail), and is complete under the operator norm $\|\cdot \|$.

 For each $g\in \Rb\Gamma$, we denote by $R_g$ the operator $\ell^2(\Gamma, \Cb)\rightarrow \ell^2(\Gamma, \Cb)$ defined by
 $R_g(x)=xg$ for $x\in \ell^2(\Gamma, \Cb)$. It is easy to see that $R_g\in \cN\Gamma$.
 Then we have the injective $\Rb$-algebra homomorphism $\Rb\Gamma\rightarrow \cN(\Gamma)$ sending $g$ to $R_{g^*}$. In this way we shall think of
 $\Rb\Gamma$ as a subalgebra of $\cN\Gamma$.

 For each $s\in \Gamma$, we also think of $s$ as the element in $\ell^2(\Gamma, \Cb)$ being $1$ at $s$ and $0$ at $t\in \Gamma\setminus \{s\}$.
 The canonical trace $\tr_{\cN\Gamma}$ on $\cN\Gamma$ is the linear functional $\cN\Gamma\rightarrow \Cb$ defined by
 $$ \tr_{\cN\Gamma}(T)=\left< Te, e\right>.$$

 An element $T\in \cN\Gamma$ is called positive if $\left<Tx, x\right>\ge 0$ for all $x\in \ell^2(\Gamma, \Cb)$. Let $T\in \cN\Gamma$ be positive. The spectral measure of $T$, is the unique Borel probability measure $\mu$ on the interval $[0, \|T\|]$ satisfying
\begin{align} \label{E-spectral measure}
 \int_0^{\|T\|}p(t)\, d\mu(t)=\tr_{\cN\Gamma}(p(T))
\end{align}
for every complex-coefficients polynomial $p$. If $\ker T=\{0\}$, then the Fuglede-Kadison determinant $\det_{\cN\Gamma}(T)$ of $T$ \cite{FK} is defined as
\begin{align} \label{E-determinant}
 \ddet_{\cN\Gamma}(T)=\exp\left(\int_{0+}^{\|T\|}\log t\, d\mu(t)\right).
\end{align}
If furthermore $f$ is invertible in $\cN\Gamma$, then
\begin{align} \label{E-determinant invertible}
 \ddet_{\cN\Gamma}(T)=\exp(\tr_{\cN\Gamma}\log T).
\end{align}

For a locally compact abelian group $X$, we denote by $\widehat{X}$ its Pontryagin dual, i.e. the locally compact abelian group of continuous group homomorphisms $X\rightarrow \Rb/\Zb$. For $f\in \Zb\Gamma$, we set $X_f=\widehat{\Zb\Gamma/\Zb\Gamma f}$. Note that one can identify $\widehat{\Zb\Gamma}$ with
$(\Rb/\Zb)^\Gamma$ naturally, with the pairing between $x\in (\Rb/\Zb)^\Gamma$ and $g\in \Zb\Gamma$ given by
$$ \left<x, g\right>=(xg^*)_e.$$
It follows that we may identify $X_f$ with $\{x\in (\Rb/\Zb)^\Gamma: xf^*=0\}$ naturally, with the pairing between $x\in X_f$ and $g+\Zb\Gamma f\in \Zb\Gamma/\Zb\Gamma f$ for $g\in \Zb\Gamma$ given by
\begin{align} \label{E-pairing}
\left<x, g+\Zb\Gamma f\right>=(xg^*)_e.
\end{align}

\section{Approximation by finite models}\label{S-approximation}

The purpose of this section is to prove:
\begin{theorem}\label{T-det vs number of fixed point}
Let $\Gamma$ be a finitely generated and residually finite infinite group. Let $\Sigma=\{\Gamma_n\}^\infty_{n=1}$ be a sequence of finite-index normal subgroups of $\Gamma$ such that $\bigcap_{n\in \Nb}\bigcup_{i\ge n}\Gamma_i=\{e\}$. Let $f\in \Zb\Gamma$ be well-balanced (Definition~\ref{defn:well-balanced}). Then
$$ \log \ddet_{\cN\Gamma} f= {\bf h}(C(\Gamma,f)) = \lim_{n\to \infty} [\Gamma:\Gamma_n]^{-1} \log \tau(C_n^f)= \lim_{n\to \infty}  [\Gamma:\Gamma_n]^{-1} \log |\Fix_{\Gamma_n}(X_f)|$$
where $\tau(C_n^f)$ is the number of spanning trees of the Cayley graph $C_n^f$ (Definition~\ref{defn:Cayley}), $\Fix_{\Gamma_n}(X_f)$ is the  fixed point set of $\Gamma_n$ in $X_f$, and $|\Fix_{\Gamma_n}(X_f)|$ is the number of connected components of $\Fix_{\Gamma_n}(X_f)$.
\end{theorem}

The following result is a special case of  \cite[Theorem 3.1]{Lyons10}. For the convenience of the reader, we give a proof here.

\begin{proposition} \label{P-tree entropy}
Let $\Gamma$ be a countable group and $f\in \Rb\Gamma$ be well-balanced. Set $\mu=-(f-f_e)/f_e$.
We have
$$ \log \ddet_{\cN\Gamma}f
= \log f_e-\sum_{k=1}^\infty \frac{1}{k}(\mu^k)_e.$$
\end{proposition}
\begin{proof} Note that $f$ is positive in $\cN\Gamma$. Let $\varepsilon>0$. Then the norm of $\frac{f_e}{f_e+\varepsilon}\mu$ in $\cN\Gamma$ is bounded above by
$\|\frac{f_e}{f_e+\varepsilon}\mu\|_1$, which in turn is strictly smaller than $1$.
Thus $\varepsilon+f=(f_e+\varepsilon)(1-\frac{f_e}{f_e+\varepsilon}\mu)$ is positive and invertible in $\cN\Gamma$.
Therefore in $\cN\Gamma$ we have
\begin{align*}
\log (\varepsilon+f)= \log (f_e+\varepsilon)-\sum_{k=1}^\infty \frac{1}{k}\cdot \left(\frac{f_e}{f_e+\varepsilon}\mu\right)^k,
\end{align*}
and the right hand side converges in norm. Thus
\begin{align*}
\log \ddet_{\cN\Gamma}(\varepsilon+f)&\overset{\eqref{E-determinant invertible}}=\tr_{\cN\Gamma} \log (\varepsilon+f)\\
&=\log (f_e+\varepsilon)-\sum_{k=1}^\infty \frac{(f_e)^k}{(f_e+\varepsilon)^kk}(\mu^k)_e.
\end{align*}
We have $\lim_{\varepsilon\to 0+}\ddet_{\cN\Gamma}(f+\varepsilon)=\ddet_{\cN\Gamma}f$ \cite[Lemma 5]{FK}. Note that $(\mu^k)_e\ge 0$ for all $k\in \Nb$. Thus
\begin{align*}
\log \ddet_{\cN\Gamma}f &=\lim_{\varepsilon\to 0+}\log \ddet_{\cN\Gamma}(f+\varepsilon) \\
&= \log f_e- \sum_{k=1}^\infty \frac{1}{k}(\mu^k)_e.
\end{align*}
\end{proof}

Recall that all graphs in this paper are allowed to have multiple edges and loops.
\begin{lemma}\label{lem:matrix-tree}
Let $G=(V,E)$ be a finite connected graph and let $\Delta_G:\R^V \to \R^V$ be the graph Laplacian: $\Delta_G x(v) = \sum_{\{v,w\}\in E} (x(v)-x(w))$. Let $\det^*(\Delta_G)$ be the product of the nonzero eigenvalues of $\Delta_G$. Then $|V|^{-1}\det^* \Delta_G =  \tau(G)$, the number of spanning trees in $G$. Moreover, if $v_0 \in V$ is any vertex, $V_0:=V \setminus \{v_0\}$,  $P_0:\R^V \to \R^{V_0}$ is the projection map and $\Delta^0_G:\R^{V_0} \to \R^{V_0}$ is defined by $\Delta^0_G = P_0\Delta_G$, then
$$\det(\Delta^0_G) = |V|^{-1} \ddet^*(\Delta_G) = \tau(G).$$
\end{lemma}

\begin{proof}
This is the Matrix-Tree Theorem (see e.g., \cite[Lemmas 13.2.3, 13.2.4]{GR}).
\end{proof}

\begin{lemma}\label{lem:correspondence}
Let $M$ be an $m\times m$ matrix with integral entries and inverse $M^{-1}$ with real entries. Let $\phi: \R^m \to \R^m$ be the corresponding linear transformation. Then the absolute value of the determinant of $M$ equals the number of integral points in $\phi([0,1)^m)$.
\end{lemma}

\begin{proof}
This is \cite[Lemma 4]{So98}.
\end{proof}

\begin{lemma}\label{lem:harmonic points}
Let $G=(V,E)$ be a finite connected graph. Let $\Delta_G:\R^V \to \R^V$ be the graph Laplacian. We also consider $\Delta_G$ as a map from $(\R/\Z)^V$ to itself. Let $X_G < (\R/\Z)^V$ be the subgroup consisting of all $x\in (\R/\Z)^V$ with $\Delta_G x = \Z^V$. Let $|X_G|$ denote the number of connected components of $X_G$. Then
$$|X_G| = |V|^{-1}\ddet^*(\Delta_G)$$
where $\ddet^*(\Delta_G)$ is the product of the non-zero eigenvalues of $\Delta_G$.
\end{lemma}

\begin{proof}
This lemma is an easy generalization of results in \cite{So98}. For completeness, we provide the details. As in Lemma \ref{lem:matrix-tree}, fix a vertex $v_0 \in V$, let $V_0:=V \setminus \{v_0\}$, think of $\R^{V_0}$ as a subspace of $\R^V$ in the obvious way, let $P_0:\R^V \to \R^{V_0}$ be the projection map and $\Delta^0_G:\R^{V_0} \to \R^{V_0}$ be defined by $\Delta^0_G = P_0\Delta_G$.

By Lemma \ref{lem:matrix-tree}, $\det \Delta_G^0 = |V|^{-1} \det^* \Delta_G$ is non-zero.
Under the standard basis of $\Rb^{V_0}$ the linear map $\Delta_{G}^0$ is represented as a matrix with integral entries. Therefore, the previous lemma implies $\det \Delta_{G}^0$ equals the number of integral points in $\Delta_{G}^0( [0,1)^{V_0} )$.

If $x \in \R^V$ has $\| x\|_\infty < (2|E|)^{-1}$ and $\Delta_G x \in \Z^V$, then $\Delta_Gx=0$ and hence $x$ is constant. Therefore, each connected component of $X_G$ is a coset of the constants. Thus $|X_G|$ is the cardinality of $X_G/Z$ where $Z<(\R/\Z)^V$ denotes the constant functions.

Given $x \in X_G$, let $\tilde{x} \in [0,1)^V$ be the unique element with $\tilde{x} + \Z^V = x$. There is a unique element $x_0 \in x + Z$ such that $\tilde{x_0}(v_0)=0$. If we let $X^0_G$ be the set of all such $\tilde{x_0}$, then $X^0_G$ is a finite set with $|X^0_G|=|X_G|$.

We claim that $X^0_G$ is precisely the set of points $y \in [0,1)^{V_0}$ with $\Delta^0_G(y) \in \Z^{V_0}$. To see this, let $\tilde{x_0} \in X^0_G$. Then $\Delta_G^0 \tilde{x_0} = P_0\Delta_G \tilde{x_0} \in \Z^{V_0}$. To see the converse, let $S:\R^V \to \R$ denote the sum function: $S(y) = \sum_{v\in V} y(v)$. Note that $S(\Delta_G(y))=0$ for any $y\in \R^V$. Therefore,  if $y \in [0,1)^{V_0}$ is any point with $\Delta_G^0(y) \in \Z^{V_0}$ then because $S(\Delta_G y)=0=\Delta_G y(v_0) + S(\Delta^0_G y)$, it must be that $\Delta_G y(v_0)$ is an integer and thus $\Delta_G y \in \Z^V$. Thus $y + \Z^V \in X_G$. Because $y(v_0)=0$, we must have $y \in X_G^0$ as claimed.

So $|X_G|$ equals $|X^0_G|$ which equals the number of points $y \in [0,1)^{V_0}$ with $\Delta_{G}^0(y) \in \Z^{V_0}$. We showed above that the later equals $\det\Delta_{G}^0 = |V|^{-1}\det^* \Delta_G$.
\end{proof}

We are ready to prove Theorem~\ref{T-det vs number of fixed point}.

\begin{proof}[Proof of Theorem \ref{T-det vs number of fixed point}]
Define the Cayley graphs $C(\Gamma,f)$ and $C_n^f=C(\Gamma/\Gamma_n,f)$ as in Definition \ref{defn:Cayley}. Then the group of harmonic mod $1$ points on $C_n^f$ is canonically isomorphic with $\Fix_{\Gamma_n}(X_f)$ \cite[Lemma 7.4]{KL11a}. So Lemmas \ref{lem:matrix-tree} and \ref{lem:harmonic points} imply $|\Fix_{\Gamma_n}(X_f)| = \tau(C_n^f)$. By \cite[Theorem 3.2]{Lyons05}, $\lim_{n\to \infty} [\Gamma:\Gamma_n]^{-1} \log \tau(C_n^f) = {\bf h}(C(\Gamma,f))$. By Proposition~\ref{P-tree entropy}, ${\bf h}(C(\Gamma,f)) = \log \ddet_{\cN\Gamma}f$. This implies the result.
\end{proof}



\section{Homoclinic Group}  \label{S-homoclinic}

For an action of a countable group $\Gamma$ on a compact metrizable group $X$ by continuous automorphisms, a point $x\in X$ is called {\it homoclinic} if
$sx$ converges to the identity element of $X$ as $\Gamma\ni s\to \infty$ \cite{LS99}. The set of all homoclinic points, denoted by $\Delta(X)$, is a $\Gamma$-invariant subgroup of $X$.

The main result of this section is the following

\begin{theorem} \label{T-dense homoclinic}
Let $\Gamma$ be a countably infinite group such that $\Gamma$ is not virtually $\Zb$ or $\Zb^2$ (i.e., does not have any finite-index normal subgroup isomorphic to $\Zb$ or $\Zb^2$). Let $f\in \Zb\Gamma$ be well-balanced. Then
\begin{enumerate}
\item The homoclinic group $\Delta(X_f)$ is dense in $X$.

\item $\alpha_f$ is mixing of all orders (with respect to the Haar probability measure of $X_f$). In particular, $\alpha_f$ is ergodic.
\end{enumerate}
\end{theorem}

\begin{remark}\label{remark:homoclinic}
When $\Gamma=\Zb^2$ and $f\in \Zb\Gamma$ is well-balanced, 
$\alpha_f$ is mixing of all orders.
We are grateful to Doug Lind for explaining this to us. Indeed, let $\Gamma=\Zb^d$ for $d\in \Nb$ and $g\in \Zb\Gamma$. If $\alpha_g$ has completely positive entropy (with respect to the Haar probability measure on $X_g$), then 
$\alpha_g$ is mixing of all orders \cite{Kaminski} \cite[Theorem 20.14]{Schmidt}.
By \cite[Theorems 20.8, 18.1, and 19.5]{Schmidt}, $\alpha_g$ has completely positive entropy exactly when $g$ has no factor in $\Zb\Gamma=\Zb[u_1^{\pm}, \dots, u_d^{\pm}]$ as a generalized cyclotomic polynomial, i.e. $g$ has no factor of the form $u_1^{m_1}\dots u_d^{m_d}h(u_1^{n_1}\dots u_d^{n_d})$ for
$m_1, \dots, m_d, n_1, \dots, n_d\in \Zb$, not all $n_1, \dots, n_d$ being $0$, and $h$ being a cyclotomic polynomial in a single variable.
When $d\ge 2$, if $g$ has a factor of such form, then $g(z_1, \dots, z_d)=0$ for some $z_1, \dots, z_d$ in the unit circle of the complex plane not being all $1$. On the other hand, when $f\in \Zb\Gamma$ is well-balanced, if $f(z_1, \dots, z_d)=0$ for some $z_1, \dots, z_d$ in the unit circle of the complex plane,
then $z_1=\dots=z_d=1$.
\end{remark}

\begin{remark} \label{R-ergodic}
Lind and Schmidt \cite{LS} showed that for any countably infinite amenable group $\Gamma$ which is not virtually $\Zb$,
if $f\in \Zb\Gamma$ is not a right zero-divisor in $\Zb\Gamma$, then $\alpha_f$ is ergodic. In particular, if $\Gamma$ is virtually
$\Zb^2$ and $f\in \Zb\Gamma$ is well-balanced, then $\alpha_f$ is ergodic.
\end{remark}

\begin{remark} \label{R-integer not ergodic}
When $\Gamma=\Zb$ and $f\in \Zb\Gamma$ is well-balanced, $\alpha_f$ is not ergodic. This follows from \cite[Theorem 6.5.(1)]{Schmidt}, and can also be seen as follows. We may identify $\Zb\Gamma$ with the Laurent polynomial ring
$\Zb[u^{\pm}]$. Since the sum of the coefficients  of $f$ is $0$, one has $f=(1-u)g$ for some $g\in \Zb[u^{\pm}]$.
It follows that $g+\Zb[u^{\pm}] f\in \Zb[u^{\pm}]/\Zb[u^{\pm}] f=\widehat{X_f}$ is fixed by the action of $\Zb=\Gamma$ (i.e., it is fixed under multiplication by $u$).
As $\Zb[u^{\pm}]$ is an integral domain and $1-u$ is not invertible in $\Zb[u^{\pm}]$, the element $g$ is not in $\Zb[u^{\pm}]f$.
Thinking of $g+\Zb[u^{\pm}]f$ as a continuous $\Cb$-valued function on $X_f$, we find that $g+\Zb[u^{\pm}]f$ has $L^2$-norm $1$ and is orthogonal to the constant functions with respect to the Haar probability measure. Thus $\alpha_f$ is not ergodic.
\end{remark}

We recall first the definition of mixing of all orders.

\begin{lemma} \label{L-mixing of all order}
Let $\Gamma$ be a countable group acting on a standard probability space $(X, \cB, \lambda)$ by measure-preserving transformations.
Let $n\in \Nb$ with $n\ge 2$. Let $s_1,\dots, s_n\in \Gamma$.
The following conditions are equivalent:
\begin{enumerate}
\item for any $A_1, \dots, A_n\in \cB$, one has $\lambda\big(\bigcap_{j=1}^ns_j^{-1}A_j\big )\to \prod_{j=1}^n\lambda(A_j)$ as $s_j^{-1}s_k\to \infty$  for all $1\le j<k\le n$.

\item for any $f_1, \dots, f_n$ in the space $L^{\infty}_{\Cb}(X, \cB, \lambda)$ of essentially bounded $\Cb$-valued $\cB$-measurable functions on $X$, one has $\lambda\big(\prod_{j=1}^ns_j(f_j)\big)\to \prod_{j=1}^n \lambda(f_j)$
as $s_j^{-1}s_k\to \infty$  for all $1\le j<k\le n$.
\end{enumerate}
If furthermore $X$ is a compact metrizable space, $\cB$ is the $\sigma$-algebra of Borel subsets of $X$, then the above conditions are also equivalent to
\begin{enumerate}
\item[(3)] for any $f_1, \dots, f_n$ in the space $C_\Cb(X)$ of $\Cb$-valued continuous functions on $X$, one has $\lambda\big(\prod_{j=1}^ns_j(f_j)\big)\to \prod_{j=1}^n \lambda(f_j)$
as $s_j^{-1}s_k\to \infty$ for all $1\le j<k\le n$.
\end{enumerate}
If furthermore $X$ is a compact metrizable abelian group, $\cB$ is the $\sigma$-algebra of Borel subsets of $X$, $\lambda$ is the Haar probability measure of $X$, and $\Gamma$ acts on $X$ by continuous automorphisms, then the above conditions are also equivalent to
\begin{enumerate}
\item[(4)] for any $f_1, \dots, f_n\in \widehat{X}$ not being all $0$, there is some finite subset $F$ of $\Gamma$ such that $\sum_{j=1}^ns_jf_j\neq 0$ for all $s_1, \dots, s_n\in \Gamma$ with
 $s_j^{-1}s_k\not\in F$ for all $1\le j<k\le n$.

\item[(5)] for any $f_1, \dots, f_n\in \widehat{X}$ with $f_1\neq 0$, there is some finite subset $F$ of $\Gamma$ such that $f_1+\sum_{j=2}^ns_jf_j\neq 0$ for all $s_2, \dots, s_n\in \Gamma$ with
 $s_j\not\in F$ for all $2\le j\le n$.
\end{enumerate}
\end{lemma}
\begin{proof} (1)$\Longleftrightarrow$(2) follows from the observation that (1) is exactly (2) when $f_j$ is the characteristic function of $A_j$ and the fact that the linear span of characteristic functions of elements in $\cB$ is dense in $L^{\infty}_\Cb(X, \cB, \lambda)$ under the essential supremum  norm $\|\cdot \|_\infty$.

(2)$\Longleftrightarrow$(3) follows from the fact that for any $f\in L^{\infty}_{\Cb}(X, \cB, \lambda)$ and $\varepsilon>0$ there exists $g\in C_\Cb(X)$ with
$\|g\|_\infty\le \|f\|_\infty$ and $\|f-g\|_2<\varepsilon$.

We identify $\Rb/\Zb$ with the unit circle $\{z\in \Cb: |z|=1\}$ naturally. Then every $g\in \widehat{X}$ can be thought of as an element in $C_\Cb(X)$. Note that the identity element $0$ in $\widehat{X}$ is the element $1$ in $C_\Cb(X)$. Then (3)$\Longleftrightarrow$(4) follows from the observation that for any $g\in \widehat{X}$, $\lambda(g)=1$ or $0$ depending on whether $g=0$ in $\widehat{X}$ or not, and the fact that
the linear span of elements in $\widehat{X}$ is dense in $C_\Cb(X)$ under the supremum norm.

(4)$\Longleftrightarrow$(5) is obvious.
\end{proof}

When the condition (1) in Lemma~\ref{L-mixing of all order} is satisfied, we say that the action is {\it (left) mixing of order $n$}.
We say that the action is {\it (left) mixing of all orders} if it is mixing of order $n$ for all $n\in \Nb$ with $n\ge 2$.

\begin{proposition} \label{P-homoclinic to mixing}
Let a countable  group $\Gamma$ act on a compact metrizable abelian group $X$ by continuous automorphisms. Suppose that the homoclinic group $\Delta(X)$ is dense in $X$. Then the action is mixing of all orders with respect to the Haar probability measure of $X$.
\end{proposition}
\begin{proof}
We verify the condition (5) in Lemma~\ref{L-mixing of all order} by contradiction.
Let $n\in \Nb$ with $n\ge 2$, and $f_1, \dots, f_n\in \widehat{X}$ with $f_1\neq 0$.
Suppose that there is a sequence $\{(s_{m, 2}, \dots, s_{m, n})\}_{m\in \Nb}$ of $(n-1)$-tuples in $\Gamma$
such that $f_1+\sum_{j=2}^ns_{m,j} f_j=0$ for all $m\in \Nb$ and $s_{m, j}\to \infty$ as $m\to \infty$ for every $2\le j\le n$.

Let $x\in \Delta(X)$. Then $f_1(x)+\sum_{j=2}^nf_j(s_{m, j}^{-1}x)=(f_1+\sum_{j=2}^ns_{m,j}f_j)(x)=0$ in $\Rb/\Zb$ for all $m\in \Nb$.
Since $s_{m, j}^{-1}x$ converges to the identity element of $X$ as $m\to \infty$ for every $2\le j\le n$, we have $f_j(s_{m, j}^{-1}x)\to 0$ in $\Rb/\Zb$ as $m\to \infty$ for every $2\le j\le n$. It follows that $f_1(x)=0$ in $\Rb/\Zb$. Since $\Delta(X)$ is dense in $X$, we get $f_1=0$, a contradiction.
\end{proof}

\begin{example} \label{E-invertible in von Neumann}
For a countable group $\Gamma$, when $g\in \Zb \Gamma$
is invertible in the group von Neumann algebra $\cN\Gamma$, $\Delta(X_g)$ is dense in $X_g$ \cite[Lemma 5.4]{CL}, and hence by Proposition~\ref{P-homoclinic to mixing} the action $\alpha_g$ is mixing of all orders with respect to the Haar probability measure of $X_g$.
\end{example}

Let $\Gamma$ be a finitely generated infinite group. Let $\mu$ be a finitely supported symmetric probability measure on $\Gamma$ such that the support of $\mu$ generates $\Gamma$. We shall think of $\mu$ as an element in $\Rb \Gamma$. We endow $\Gamma$ with the word length associated to the support of $\mu$.

By the Cauchy-Schwarz inequality for any $s\in \Gamma$ and $n\in \Nb$ one has $(\mu^{2n})_s=(\mu^n\cdot (\mu^n)^*)_s\le (\mu^n\cdot (\mu^n)^*)_e=(\mu^{2n})_e$.
Also note that for any $s\in \Gamma$ and $n\in \Nb$ one has $(\mu^{2n+1})_s=(\mu^{2n}\cdot \mu)_s\le \|\mu^{2n}\|_\infty$. It follows that
$\sum_{k=0}^{\infty} (\mu^k)_e<+\infty$ if and only if $\sum_{k=0}^{\infty}\|\mu^k\|_\infty<+\infty$.

Now assume that $\Gamma$ is not virtually $\Zb$ or $\Zb^2$.
 By  a result of Varopoulos \cite{Varopoulos} \cite[Theorem 2.1]{Fu02}
\cite[Theorem 3.24]{Woess} one has $\sum_{k=0}^{\infty} (\mu^k)_e<+\infty$, and hence $\sum_{k=0}^{\infty}\|\mu^k\|_\infty<+\infty$.
Thus we have the element $\sum_{k=0}^\infty \mu^k$ in $\ell^\infty(\Gamma)$. Let $\varepsilon>0$. Take $m\in \Nb$ such that
$\sum_{k=m+1}^\infty\|\mu^k\|_\infty<\varepsilon$. For each $s\in \Gamma$ with word length at least $m+1$, one has
$$|(\sum_{k=0}^\infty \mu^k)_s|=|(\sum_{k=m+1}^\infty \mu^k)_s|\le \sum_{k=m+1}^\infty\|\mu^k\|_\infty<\varepsilon.$$
Therefore $\sum_{k=0}^\infty \mu^k$ lies in the space $C_0(\Gamma)$ of $\Rb$-valued functions on $\Gamma$ vanishing at $\infty$. Since the support of $\mu$ is symmetric and generates $\Gamma$, one has $(\sum_{k=0}^\infty \mu^k)_s>0$ for every $s\in \Gamma$.

Now let $f\in \Rb\Gamma$ be well-balanced.
Set $\mu=-(f-f_e)/f_e$. Then $\mu$ is a finitely supported symmetric probability measure on $\Gamma$ and the support of $\mu$ generates $\Gamma$. Set $\omega=f_e^{-1}\sum_{k=0}^\infty\mu^k\in C_0(\Gamma)$. We have $f=f_e(1-\mu)$, and hence
$$ f\omega=(1-\mu)\sum_{k=0}^\infty\mu^k=1$$
and
$$ \omega f=(\sum_{k=0}^\infty\mu^k)(1-\mu)=1.$$

Note that the space $C_0(\Gamma)$ is not closed under convolution. The above identities show that $\omega$ is a formal inverse of $f$ in $C_0(\Gamma)$. Now we show that $f$ has no other formal inverse in $C_0(\Gamma)$.

\begin{lemma} \label{L-associative}
Let $g\in C_0(\Gamma)$ such that $gf\in \ell^1(\Gamma)$. Then
$$ (gf)\omega=g.$$
\end{lemma}
\begin{proof} In the Banach space $\ell^\infty(\Gamma)$ we have
\begin{align*}
(gf) \omega &= \lim_{m\to \infty}(gf)(f_e^{-1}\sum_{k=0}^m\mu^k) = \lim_{m\to \infty}g(f f_e^{-1}\sum_{k=0}^m\mu^k) = \lim_{m\to \infty}g((1-\mu)\sum_{k=0}^m\mu^k) \\
&= \lim_{m\to \infty}g(1-\mu^{m+1}) = g-\lim_{m\to \infty}g\mu^{m+1}.
\end{align*}
Let $\varepsilon>0$. Take a finite set $F\subset\Gamma$ such that $\|g|_{\Gamma \setminus F}\|_\infty<\varepsilon$. Write $g$ as $u+v$ for $u, v\in \ell^\infty(\Gamma)$ such that $u$ has support contained in $F$ and $v$ has support contained in $\Gamma\setminus F$. For each $m\in \Nb$ we have
\begin{align*}
\|g\mu^{m+1}\|_\infty &\le \|u\mu^{m+1}\|_\infty+\|v\mu^{m+1}\|_\infty \le \|u\|_1\|\mu^{m+1}\|_\infty+\|v\|_\infty\|\mu^{m+1}\|_1\\
&\le \|g\|_\infty \cdot |F|\cdot \|\mu^{m+1}\|_\infty+ \varepsilon.
\end{align*}
Letting $m\to \infty$, we get $\limsup_{m\to \infty}\|g\mu^{m+1}\|_\infty\le \varepsilon$. Since $\varepsilon$ is an arbitrary positive number, we get
$\limsup_{m\to \infty}\|g\mu^{m+1}\|_\infty=0$ and hence $\lim_{m\to \infty}g\mu^{m+1}=0$. It follows that $(gf)\omega=g$ as desired.
\end{proof}

\begin{corollary} \label{C-inverse}
Let $g\in C_0(\Gamma)$ such that $gf=1$. Then $g=\omega$.
\end{corollary}
\begin{proof} By Lemma~\ref{L-associative} we have
$$ g=(gf)\omega=\omega.$$
\end{proof}

Denote by $Q$ the natural quotient map $\ell^\infty(\Gamma)\rightarrow (\Rb/\Zb)^\Gamma$. We assume furthermore that  $f\in \Zb\Gamma$.

\begin{lemma} \label{L-homoclinic group}
We have
$$ \Delta(X_f)=Q(\Zb\Gamma \omega).$$
\end{lemma}
\begin{proof} Let $h\in \Zb\Gamma$. Then
$$ (h\omega)f=h(\omega f)=h\in \Zb\Gamma,$$
and hence $Q(h\omega)\in X_f$. Since $h\omega \in C_0(\Gamma)$, one has $Q(h\omega)\in \Delta(X_f)$. Thus $Q(\Zb\Gamma \omega)\subset\Delta(X_f)$.

Now let $x\in \Delta(X_f)$. Take $\tilde{x}\in C_0(\Gamma)$ such that $Q(\tilde{x})=x$. Then $\tilde{x} f\in C_0(\Gamma)\cap \ell^\infty(\Gamma, \Zb)=\Zb\Gamma$.
Set $h=\tilde{x}f\in \Zb\Gamma$.
By Lemma~\ref{L-associative} we have $\tilde{x}=h\omega$. Therefore
$x=Q(\tilde{x})=Q(h\omega)$, and hence $\Delta(X_f)\subset Q(\Zb\Gamma \omega)$ as desired.
\end{proof}

\begin{corollary} \label{C-module}
As a left $\Zb\Gamma$-module, $\Delta(X_f)$ is isomorphic to $\Zb\Gamma/\Zb\Gamma f$.
\end{corollary}
\begin{proof} By Lemma~\ref{L-homoclinic group} we have the surjective left $\Zb\Gamma$-module map $\Phi: \Zb\Gamma\rightarrow \Delta(X_f)$ sending $h$ to $Q(h\omega)$. Then it suffices to show $\ker \Phi=\Zb\Gamma f$.

If $h\in \Zb\Gamma f$, say $h=gf$ for some $g\in \Zb\Gamma$, then
$$ Q(h\omega)=Q((gf)\omega)=Q(g(f\omega))=Q(g)=0.$$
Thus $\Zb\Gamma \subset\ker \Phi$.

Let $h\in \ker \Phi$. Then $h\omega \in C_0(\Gamma)\cap \ell^\infty(\Gamma, \Zb)=\Zb\Gamma$. Set $g=h\omega \in \Zb\Gamma$. Then
$$ gf=(h\omega)f=h(\omega f)=h.$$
Thus $\ker\Phi\subset\Zb\Gamma f$.
\end{proof}

We are ready to prove Theorem~\ref{T-dense homoclinic}.

\begin{proof}[Proof of Theorem~\ref{T-dense homoclinic}] The assertion (2) follows from the assertion (1) and Proposition~\ref{P-homoclinic to mixing}.

To prove the assertion (1), by Pontryagin duality it suffices to show any $\varphi \in \widehat{X_f}$ vanishing on $\Delta(X_f)$ is $0$. Thus let $\varphi \in \widehat{X_f}=\Zb\Gamma/\Zb\Gamma f$ vanishing on $\Delta(X_f)$. Say, $\varphi=g+\Zb\Gamma f$ for some $g\in \Zb\Gamma$. For each $h\in \Zb\Gamma$, in $\Rb/\Zb$ one has
$$ 0=\left< Q(h\omega), \varphi\right>\overset{\eqref{E-pairing}}=((h\omega) g^*)_e+\Zb=(h(\omega g^*))_e+\Zb,$$
and hence $(h(\omega g^*))_e\in \Zb$. Taking $h=s$ for all $s\in \Gamma$, we conclude that $\omega g^*\in C_0(\Gamma)\cap \ell^\infty(\Gamma, \Zb)=\Zb\Gamma$.
Set $v=\omega g^*\in \Zb\Gamma$. Then
$$ fv=f(\omega g^*)=(f\omega)g^*=g^*,$$
and hence
$$ g=v^*f^*=v^*f\in \Zb\Gamma f.$$
Therefore $\varphi=g+\Zb\Gamma f=0$ as desired.
\end{proof}

\section{Periodic Points} \label{S-periodic}

Throughout this section we let $\Gamma$ be a finitely generated residually finite infinite group, and let $\Sigma=\{\Gamma_n\}^\infty_{n=1}$ be a sequence of finite-index normal subgroups of $\Gamma$ such that $\bigcap_{n\in \Nb}\bigcup_{i\ge n}\Gamma_i=\{e\}$.

For a compact metric space $(X, \rho)$, recall that the Hausdorff distance between two nonempty closed subsets $Y$ and $Z$ of $X$ is defined as
$$ {\rm dist_H}(Y, Z):=\max(\max_{y\in Y}\min_{z\in Z}\rho(y, z), \max_{z\in Z}\min_{y\in Y}\rho(z, y)).$$

For $f\in \Zb\Gamma$ we denote by $\Fix_{\Gamma_n}(X_f)$ the group of fixed points of $\Gamma_n$ in $X_f$.
The main result of this section is

\begin{theorem} \label{T-dense periodic points}
Let $f\in \Zb\Gamma$ be well-balanced.
Let $\rho$ be a
compatible metric on $X_f$. We have $\Fix_{\Gamma_n}(X_f)\rightarrow X_f$ under the Hausdorff distance when $n\to \infty$.
\end{theorem}

Denote by $\pi_n$ the natural ring homomorphism $\Rb\Gamma \rightarrow \Rb(\Gamma/\Gamma_n)$. Let $X_{\pi_n(f)}$ be $\widehat{\Zb(\Gamma/\Gamma_n)/\Zb(\Gamma/\Gamma_n)\pi_n(f)}$, which is the additive group of all maps $x:\Gamma/\Gamma_n \to \R/\Z$ satisfying $x\pi_n(f) =0$, i.e., $\sum_{s\in \Gamma} x_{ts\Gamma_n} f_{s^{-1}} = \Z$ for every $t\in \Gamma$.


\begin{lemma} \label{L-vanishing}
Let $f\in \Zb\Gamma$ and  $\varphi=g+\Zb\Gamma f\in \Zb\Gamma/\Zb\Gamma f$ for some $g\in \Zb\Gamma$. Let  $n\in \Nb$. Then $\varphi$ vanishes on $\Fix_{\Gamma_n}(X_f)$ if and only if $\pi_n(g)\in \Zb(\Gamma/\Gamma_n)\pi_n(f)$.
\end{lemma}
\begin{proof}
By \cite[Lemma 7.4]{KL11a} we have a compact group isomorphism $\Phi_n:X_{\pi_n(f)} \rightarrow \Fix_{\Gamma_n}(X_f)$ defined by $(\Phi_n(x))_s=x_{s\Gamma_n}$ for all $x\in X_{\pi_n(f)}$ and $s\in \Gamma$.

For each $x\in X_{\pi_n(f)}$, in $\Rb/\Zb$ we have
\begin{align*}
\left<\Phi_n(x), \varphi\right>&\overset{\eqref{E-pairing}}=(\Phi_n(x)g^*)_e =\sum_{s\in \Gamma} (\Phi_n(x))_sg_s  =\sum_{s\in \Gamma} x_{s\Gamma_n}g_s \\
&= \sum_{s\Gamma_n \in \Gamma/\Gamma_n} x_{s\Gamma_n}\pi_n(g)_{s\Gamma_n} =\left<x, \pi_n(g)+\Zb(\Gamma/\Gamma_n)\pi_n(f)\right>.
\end{align*}
Thus $\varphi$ vanishes on $\Fix_{\Gamma_n}(X_f)$ iff the element $\pi_n(g)+\Zb(\Gamma/\Gamma_n)\pi_n(f)$ in $\Zb(\Gamma/\Gamma_n)/\Zb(\Gamma/\Gamma_n)\pi_n(f)$ vanishes on $X_{\pi_n(f)}$, iff $\pi_n(g)\in \Zb(\Gamma/\Gamma_n)\pi_n(f)$.
\end{proof}

For the next lemma, recall that if $S \subset \Gamma \setminus \{e\}$ is a symmetric finite generating set then $C(\Gamma,S)$, the Cayley graph of $\Gamma$ with respect to $S$, has vertex set $\Gamma$ and edge set  $\{ \{g,gs\} \}_{g\in \Gamma,s\in S}$. Similarly, $C(\Gamma/\Gamma_n, \pi_n(S))$ has vertex set $\Gamma/\Gamma_n$ and edge set $\{\{g\Gamma_n, gs\Gamma_n\}:~g\in \Gamma, s\in S\}$. A subset $A \subset \Gamma$ is identified with the induced subgraph of $C(\Gamma,S)$ which has vertex set $A$ and contains every edge of $C(\Gamma,S)$ with endpoints in $A$. Similarly, a subset $A \subset \Gamma/\Gamma_n$ induces a subgraph of $C(\Gamma/\Gamma_n,\pi_n(S))$. Thus we say that a subset $A \subset \Gamma$ (or $A \subset \Gamma/\Gamma_n$) is {\em connected} if its induced subgraph is connected.

\begin{lemma} \label{L-connectedness in Cayley graphs}
Let $S\subset \Gamma\setminus \{e\}$ be a finite symmetric generating set of $\Gamma$.
Let $A\subset\Gamma$ be finite. Then there exists a finite set $B\subset\Gamma$ containing $A$ such that when $n\in \Nb$ is large enough, in the Cayley graph $C(\Gamma/\Gamma_n, \pi_n(S))$ the set  $(\Gamma/\Gamma_n)\setminus \pi_n(B)$ is connected.
\end{lemma}

\begin{proof}
We claim that there exists a connected finite set $B \supset A \cup \{e\}$ such that every connected component of $C(\Gamma,S)\setminus B$ is infinite. To see this, let $A'\subset\Gamma$ be a finite connected set such that $A'\supset A\cup \{e\}$. For any connected component $\cC$ of $C(\Gamma,S) \setminus A'$, taking a path in $C(\Gamma, S)$ from some point in $\cC$ to some point in $A'$, we note that the last point of this path in $\cC$ must lie in $A'S$, whence $\cC \cap A'S\neq \emptyset$. It follows that $C(\Gamma,S) \setminus A'$ has only finitely many connected components. Denote by $B$ the union of $A'$ and all the finite connected components of $C(\Gamma,S) \setminus A'$. Then $B$ is finite, contains $A\cup \{e\}$, and is connected.
Furthermore, the connected components of $C(\Gamma,S) \setminus B$ are exactly the infinite connected components of $C(\Gamma,S) \setminus A'$, whence are all infinite.


For each $t\in BS\setminus B$, since the connected component of $C(\Gamma,S) \setminus B$ containing $t$ is infinite, we can take a path $\gamma_t$ in $C(\Gamma,S) \setminus B$ from $t$ to some element in $\Gamma\setminus (B(\{e\}\cup S)^2B^{-1})$.
Let $n \in \Nb$ be sufficiently large so that $\pi_n$ is injective on $B\cup (B(\{e\}\cup S)^2B^{-1}) \cup \bigcup_{t\in BS\setminus B}\gamma_t$.

An argument similar to that in the first paragraph of the proof shows that every connected component $\cC$ of $C(\Gamma/\Gamma_n,\pi_n(S)) \setminus \pi_n(B)$ has nonempty intersection with $\pi_n(BS)$, and hence contains $\pi_n(t)$ for some $t\in BS\setminus B$. Then $\cC$ contains $\pi_n(\gamma_t)$, whence contains $\pi_n(g)$ for some $g\in \gamma_t\setminus (B(\{e\}\cup S)^2B^{-1})$. Since $\pi_n$ is injective on $\gamma_t\cup (B(\{e\}\cup S)^2B^{-1})$, we have
$\pi_n(g)\not \in \pi_n(B(\{e\}\cup S)^2B^{-1})$, equivalently,
$g\pi_n(B\cup BS)\cap \pi_n(B\cup BS)=\emptyset$.


 List all the connected components of $C(\Gamma/\Gamma_n,\pi_n(S)) \setminus \pi_n(B)$ as $\cC_0,\cC_1,\ldots, \cC_k$.
 We may assume that $|\cC_0| = \min_{0\le j\le k} |\cC_j|$. Take $g_0\in \Gamma$ with $\pi_n(g_0)\in \cC_0$ and $g_0\pi_n(B\cup BS)\cap \pi_n(B\cup BS)=\emptyset$. Since $B$ is connected, $B\cup BS$ is connected. Thus $g_0\pi_n(B\cup BS)$ is connected and disjoint from $\pi_n(B)$.
 Therefore $g_0\pi_n(B\cup BS)$ is  contained in one of the connected components of $C(\Gamma/\Gamma_n,\pi_n(S)) \setminus \pi_n(B)$.
 As $\pi_n(g_0)\in (g_0\pi_n(B\cup BS))\cap \cC_0$, we get $g_0\pi_n(B\cup BS)\subset\cC_0$.


Since $\Gamma$ acts on $C(\Gamma/\Gamma_n,\pi_n(S))$ by left translation, the connected components of $C(\Gamma/\Gamma_n,\pi_n(S)) \setminus g_0\pi_n(B)$
are $g_0\cC_0, g_0\cC_1, \dots, g_0\cC_k$.



Suppose that $g_0\cC_0 \cap \pi_n(B) = \emptyset$. Then $g_0\cC_0$ must be contained in one of the connected components of $C(\Gamma/\Gamma_n,\pi_n(S)) \setminus \pi_n(B)$. Since $\cC_0\cap \pi_n(BS)\neq \emptyset$, one has $g_0\cC_0\cap g_0\pi_n(BS)\neq \emptyset$. Because $g_0\pi_n(BS)\subset\cC_0$, we have
$g_0\cC_0\cap \cC_0\neq \emptyset$. Therefore $g_0\cC_0\subset\cC_0$.
Because $|g_0\cC_0|=|\cC_0|$, this implies $g_0\cC_0=\cC_0$. But this contradicts the fact that $g_0\pi_n(B) \subset\cC_0$ but $g_0\pi_n(B) \cap g_0\cC_0 = \emptyset$.
So $g_0\cC_0 \cap \pi_n(B) \ne \emptyset$.

Since $\pi_n(B\cup BS)$ is connected and disjoint from $g_0\pi_n(B)$, it is contained in one of the connected components of $C(\Gamma/\Gamma_n,\pi_n(S)) \setminus g_0\pi_n(B)$. Because $g_0\cC_0 \cap \pi_n(B) \ne \emptyset$, we get $\pi_n(B\cup BS)\subset g_0\cC_0$.

Suppose that $k>0$. Since $\cC_k$ is disjoint from $\cC_0$ and $g_0\pi_n(B)\subset\cC_0$,
$\cC_k$ is disjoint from $g_0\pi_n(B)$. Thus $\cC_k$  is contained in one of the connected components of $C(\Gamma/\Gamma_n,\pi_n(S)) \setminus g_0\pi_n(B)$. Because $\cC_k$ has nonempty intersection with $\pi_n(BS)$, and $\pi_n(B\cup BS)\subset g_0\cC_0$, we get
$\cC_k\cup \pi_n(B) \subset g_0\cC_0$. Therefore $|\cC_0|=|g_0\cC_0|>|\cC_k|$
which contradicts the fact that $|\cC_0| = \min_{0\le j\le k} |\cC_j|$. Thus $k=0$; i.e., $C(\Gamma/\Gamma_n,\pi_n(S)) \setminus \pi_n(B)$ is connected.
\end{proof}


\begin{lemma} \label{L-bounded}
Let $f\in\Rb\Gamma$ be well-balanced and $g\in \Rb\Gamma$. Then there exists $C>0$ such that if $\pi_n(g)=h\pi_n(f)$ for some $n\in \Nb$ and $h\in \Rb(\Gamma/\Gamma_n)$, then
$$\max_{}\{h_{s\Gamma_n}:~s\Gamma_n\in \Gamma/\Gamma_n\} -\min_{}\{h_{s\Gamma_n}:~s\Gamma_n\in \Gamma/\Gamma_n\}\le C.$$
\end{lemma}
\begin{proof} Set $\mu=-(f-f_e)/f_e$. Then $\mu$ is a symmetric finitely supported probability measure on $\Gamma$.
Denote by $S$ and $K$ the supports of $\mu$ and $g$ respectively. Set $a=\min_{s\in S} \mu_s$.

Suppose that $\pi_n(g)=h\pi_n(f)$ for some $n\in \Nb$ and $h\in \Rb(\Gamma/\Gamma_n)$. Denote by $W$ the set of $s\Gamma_n\in \Gamma/\Gamma_n$ satisfying
$h_{s\Gamma_n}=\min_{}\{h_{t\Gamma_n}:~t\Gamma_n\in \Gamma/\Gamma_n\}$. If $s\Gamma_n\in W\setminus K\Gamma_n$, then from
$$ 0=(\pi_n(g))_{s\Gamma_n}=(h\pi_n(f))_{s\Gamma_n}=f_e(h_{s\Gamma_n}-\sum_{t\in S}h_{st^{-1}\Gamma_n}\mu_t)$$
we conclude that $st\Gamma_n\in W$ for all $t\in S$. Since $S$ generates $\Gamma$, it follows that there exists some $s_{\min} \in K$
satisfying $h_{s_{\min}\Gamma_n}=\min_{}\{ h_{t\Gamma_n}:~ t\Gamma_n\in \Gamma/\Gamma_n\}$. Similarly, there exists some $s_{\max}\in K$
satisfying $h_{s_{\max}\Gamma_n}=\max_{}\{ h_{t\Gamma_n}:~ t\Gamma_n\in \Gamma/\Gamma_n\}$.

Now we show by induction on the word length $|t|$ of $t\in \Gamma$ with respect to $S$ that one has $h_{s_{\min}t\Gamma_n}\le h_{s_{\min}\Gamma_n}+|t|\cdot \frac{\|g\|_1}{f_ea^{|t|}}$ for all $t\in \Gamma$.
This is clear when $|t|=0$, i.e. $t=e$. Suppose that this holds for all $t\in \Gamma$ with $|t|\le k$. Let $t\in \Gamma$ with $|t|=k+1$. Then
$t=t_1s_1$ for some $t_1\in \Gamma$ with $|t_1|=k$ and some $s_1\in S$. Note that
\begin{align*}
-\|g\|_1&\le (\pi_n(g))_{s_{\min}t_1\Gamma_n}=(h\pi_n(f))_{s_{\min}t_1\Gamma_n}\\
&=f_e\left(h_{s_{\min}t_1\Gamma_n}-h_{s_{\min}t_1s_1\Gamma_n}\mu_{s_1^{-1}}-\sum_{s\in S\setminus \{s_1\}}h_{s_{\min}t_1s\Gamma_n}\mu_{s^{-1}}\right),
\end{align*}
and hence
\begin{align*}
h_{s_{\min}t\Gamma_n}&=
h_{s_{\min}t_1s_1\Gamma_n}\\
&\le \left(\frac{\|g\|_1}{f_e}+h_{s_{\min}t_1\Gamma_n}-\sum_{s\in S\setminus \{s_1\}}h_{s_{\min}t_1s\Gamma_n}\mu_{s^{-1}}\right)/\mu_{s_1^{-1}}\\
&\le \left(\frac{\|g\|_1}{f_e}+h_{s_{\min}\Gamma_n}+k\cdot \frac{\|g\|_1}{f_ea^k}-\sum_{s\in S\setminus \{s_1\}}h_{s_{\min}\Gamma_n}\mu_{s^{-1}}\right)/\mu_{s_1^{-1}}\\
&= \left(\frac{\|g\|_1}{f_e}+h_{s_{\min}\Gamma_n}\mu_{s_1^{-1}}+k\cdot \frac{\|g\|_1}{f_ea^k}\right)/\mu_{s_1^{-1}}\le h_{s_{\min}\Gamma_n}+|t|\cdot \frac{\|g\|_1}{f_ea^{|t|}}.
\end{align*}
This finishes the induction.

Set $m=\max_{s\in K^{-1}K} |s|$. Taking $t=s_{\min}^{-1}s_{\max}$ in above we get
\begin{align*}
h_{s_{\max}\Gamma_n}-h_{s_{\min}\Gamma_n}\le |s_{\min}^{-1}s_{\max}|\cdot \frac{\|g\|_1}{f_ea^{|s_{\min}^{-1}s_{\max}|}}\le \frac{m \|g\|_1}{f_e a^m}.
\end{align*}
Now we may set $C=\frac{m \|g\|_1}{f_e a^m}$.
\end{proof}

\begin{lemma} \label{L-connected component to left module}
Let $f\in \Zb\Gamma$ be well-balanced and $g\in \Zb\Gamma$. Then $\pi_n(g)\in \Zb(\Gamma/\Gamma_n)\pi_n(f)$ for all $n\in \Nb$ if and only if $g\in \Zb\Gamma f$.
\end{lemma}
\begin{proof}
The ``if'' part is obvious. Suppose that $\pi_n(g)\in \Zb(\Gamma/\Gamma_n)\pi_n(f)$ for all $n\in \Nb$. For each $n\in \Nb$, take $h_n\in \Zb(\Gamma/\Gamma_n)$ such that $\pi_n(g)=h_n\pi_n(f)$. By Lemma~\ref{L-bounded} there exists some $C\in \Nb$ such that
$$\max_{s\Gamma_n\in \Gamma/\Gamma_n}(h_n)_{s\Gamma_n}-\min_{s\Gamma_n\in \Gamma/\Gamma_n}(h_n)_{s\Gamma_n}\le C$$
for all $n\in \Nb$.

Let $S$ be the set of all $s\in \Gamma\setminus \{e\}$ with $f_s\neq 0$. Denote  by $A_1$ the support of $g$. By Lemma~\ref{L-connectedness in Cayley graphs} we can find finite subsets $A_2, \dots, A_{C+1}$ of $\Gamma$ and $N\in \Nb$ such that for any $n\ge N$ and $1\le j\le C$, one has $A_j(\{e\}\cup S)\subset A_{j+1}$ and in the Cayley graph $C(\Gamma/\Gamma_n, \pi_n(S))$ the set
 $(\Gamma/\Gamma_n)\setminus \pi_n(A_{j+1})$ is connected.

Let $n\ge N$. Set $a_j=\max_{}\{(h_n)_{s\Gamma_n}:~s\Gamma_n\in (\Gamma/\Gamma_n)\setminus \pi_n(A_j)\}$ and $b_j=\min_{}\{(h_n)_{s\Gamma_n}:~s\Gamma_n\in (\Gamma/\Gamma_n)\setminus \pi_n(A_j)\}$ for all $1\le j\le C+1$. We shall show by induction that
$$ a_j-b_j\le C+1-j$$
for all $1\le j\le C+1$. This is trivial when $j=1$. Suppose that $a_j-b_j\le C+1-j$ for some $1\le j\le C$. If $a_{j+1}<a_j$,
then
$a_{j+1}-b_{j+1}<a_j-b_j\le C+1-j$ and hence $a_{j+1}-b_{j+1}\le C+1-(j+1)$. Thus we may assume that $a_{j+1}=a_j$.
 Denote by $W$ the set of elements in $(\Gamma/\Gamma_n)\setminus \pi_n(A_{j+1})$ at which $h_n$ takes the value $a_{j+1}$.
 Let $t\Gamma_n\in W$. Since $\pi_n(g)$ takes value $0$ at $t\Gamma_n$, we have
$$ f_e (\pi_n(h))_{t\Gamma_n}=\sum_{s\in S}(-f_s) (\pi_n(h))_{ts\Gamma_n}.$$
Note that $ts\Gamma_n\in  (\Gamma/\Gamma_n)\setminus \pi_n(A_j)$ for all $s\in S$. Thus $(\pi_n(h))_{ts\Gamma_n}\le a_j=a_{j+1}=(\pi_n(h))_{t\Gamma_n}$ for all $s\in S$, and hence $(\pi_n(h))_{t\Gamma_n}=(\pi_n(h))_{ts\Gamma_n}$ for all $s\in S$. Therefore, if for some $s\in S$ one has $ts\Gamma_n\in (\Gamma/\Gamma_n)\setminus \pi_n(A_{j+1})$, then $ts\Gamma_n\in W$. Take $t_1\Gamma_n, t_2 \Gamma_n \in \Gamma/\Gamma_n$.
By our choice of $A_{j+1}$ we have a path in $(\Gamma/\Gamma_n)\setminus \pi_n(A_{j+1})$ connecting $t_1\Gamma_n$ and $t_2\Gamma_n$. Therefore
$t_1\Gamma_n\in W \Leftrightarrow t_2\Gamma_n \in W$, whence $a_{j+1}-b_{j+1}=0\le C+1-(j+1)$.

Now we have that $h_n$ is a constant function on $(\Gamma/\Gamma_n)\setminus \pi_n(A_{C+1})$. Replacing $h_n$ by the difference of $h_n$ and a suitable constant function,
we may assume that $h_n$ is $0$ on $(\Gamma/\Gamma_n)\setminus \pi_n(A_{C+1})$. Then $\|h_n\|_\infty\le C$.

Passing to a subsequence of $\{\Gamma_n\}_{n\in \Nb}$ if necessary, we may assume that $h_n(s\Gamma_n)$ converges to some integer $h_s$ as $n\to \infty$ for every $s\in \Gamma$. Then $h_s=0$ for all $s\in \Gamma\setminus A_{C+1}$. Thus $h\in \Zb\Gamma$.
Note that
$$ (hf)_s=\lim_{n\to \infty}(h_n\pi_n(f))_{s\Gamma_n}=\lim_{n\to \infty}(\pi_n(g))_{s\Gamma_n}=g_s$$
for each $s\in \Gamma$ and hence $g=hf\in \Zb\Gamma f$. This proves the ``only if'' part.
\end{proof}

We are ready to prove Theorem~\ref{T-dense periodic points}.

\begin{proof}[Proof of Theorem~\ref{T-dense periodic points}]
Since $X_f$ is compact, the set of nonempty closed subsets of $X_f$ is a compact space under the Hausdorff distance \cite[Theorem 7.3.8]{BBI}. Thus, passing to a subsequence of $\{\Gamma_n\}_{n\in \Nb}$ if necessary, we may assume that $\Fix_{\Gamma_n}(X_f)$ converges to some nonempty closed subset $Y$ of $X_f$ under the Hausdorff distance as $n\to \infty$. A point $x\in X_f$ is in $Y$ exactly when $x=\lim_{n\to \infty} x_n$ for some $x_n\in \Fix_{\Gamma_n}(X_f)$ for each $n\in \Nb$. It follows easily that $Y$ is a closed subgroup of $X_f$. Thus, by Pontryagin duality it suffices to show that the only $\varphi\in \widehat{X_f}=\Zb\Gamma/\Zb\Gamma f$ vanishing on $Y$ is $0$. Let $\varphi\in \widehat{X_f}=\Zb\Gamma/\Zb\Gamma f$ vanishing on $Y$. Say, $\varphi=g+\Zb\Gamma f$ for some $g\in \Zb\Gamma$.

Let $U$ be a small neighborhood of $0$ in $\Rb/\Zb$ such that the only subgroup of $\Rb/\Zb$ contained in $U$ is $\{0\}$. Since $\Fix_{\Gamma_n}(X_f)$ converges to $Y$ under the Hausdorff distance, we have $\left<\Fix_{\Gamma_n}(X_f), \varphi\right>\subset U$ for all sufficiently large $n$. Note that
$\left<\Fix_{\Gamma_n}(X_f), \varphi\right>$ is a subgroup of $\Rb/\Zb$. By our choice of $U$, we see that $\varphi$ vanishes on $\Fix_{\Gamma_n}(X_f)$ for all sufficiently large $n$. Without loss of generality, we may assume that $\varphi$ vanishes on $\Fix_{\Gamma_n}(X_f)$ for all $n\in \Nb$.
From Lemmas~\ref{L-vanishing} and \ref{L-connected component to left module} we get $g\in \Zb\Gamma f$ and hence $\varphi=0$ as desired.
\end{proof}


\section{Sofic Entropy}\label{sec:sofic}
The purpose of this section is to review sofic entropy theory. To be precise, we use Definitions 2.2 and 3.3 of \cite{KL11b} to define entropy. So let $\Gamma$ act by homeomorphisms on a compact metrizable space $X$. Suppose the action preserves a Borel probability measure $\lambda$. Let $\Sigma:=\{\Gamma_n\}^\infty_{n=1}$ be a sequence of finite-index normal subgroups of $\Gamma$ such that $\bigcap_{n\in \N}\bigcup_{i\ge n}\Gamma_i=\{e\}$.


Let $\rho$ be a continuous pseudo-metric on $X$. For $n \in \N$, define, on the space $\Map(\Gamma/\Gamma_n,X)$ of all maps from $\Gamma/\Gamma_n$ to $X$ the pseudo-metrics
\begin{eqnarray*}
\rho_2(\phi,\psi)&:=& \left([\Gamma:\Gamma_n]^{-1}\sum_{s\Gamma_n \in \Gamma/\Gamma_n} \rho(\phi(s \Gamma_n), \psi(s\Gamma_n) )^2 \right)^{1/2},\\
\rho_\infty(\phi,\psi)&:=& \sup_{s\Gamma_n \in \Gamma/\Gamma_n} \rho(\phi(s \Gamma_n), \psi(s\Gamma_n)).
\end{eqnarray*}


\begin{definition}
Let $W\subset \Gamma$ be finite and nonempty and $\delta>0$. Define $\Map(W,\delta,\Gamma_n)$ to be the set of all maps $\phi: \Gamma/\Gamma_n \to X$ such that $\rho_2(\phi \circ s, s \circ \phi) \le \delta$ for all $s\in W$.

Given a finite set $L$ in the space $C(X)$ of continuous $\Rb$-valued functions on $X$, let  $\Map_\lambda(W,L,\delta,\Gamma_n) \subset \Map(W,\delta,\Gamma_n)$ be the subset of maps $\phi: \Gamma/\Gamma_n \to X$ such that $|\phi_*U_n(p)-\lambda(p)| \le  \delta$ for all $p\in L$, where $U_n$ denotes the uniform probability measure on $\Gamma/\Gamma_n$.
\end{definition}

\begin{definition}\label{defn:separating}
Let $(Z,\rho_Z)$ be a pseudo-metric space. A set $Y \subset Z$ is {\em $(\rho_Z,\epsilon)$-separating} if $\rho_Z(y_1,y_2) >\epsilon$ for every $y_1\ne y_2 \in Y$. If $\rho_Z$ is understood, then we simply say that $Y$ is {\em $\epsilon$-separating}. Let $N_\epsilon(Z,\rho_Z)$ denote the largest cardinality of a $(\rho_Z,\epsilon)$-separating subset of $Z$.
\end{definition}

Define
\begin{eqnarray*}
h_{\Sigma,2}(\rho)&:=&\sup_{\epsilon>0} \inf_{W \subset \Gamma}  \inf_{\delta>0} \limsup_{n\to\infty} [\Gamma:\Gamma_n]^{-1} \log N_\epsilon(\Map(W,\delta,\Gamma_n),\rho_2)\\
h_{\Sigma,\lambda,2}(\rho)&:=&\sup_{\epsilon>0} \inf_{W \subset \Gamma} \inf_{L\subset C(X)} \inf_{\delta>0} \limsup_{n\to\infty} [\Gamma:\Gamma_n]^{-1} \log N_\epsilon(\Map_\lambda(W,L,\delta,\Gamma_n),\rho_2).
\end{eqnarray*}
Similarly, define $h_{\Sigma,\infty}(\rho)$ and $h_{\Sigma,\lambda,\infty}(\rho)$ by replacing $\rho_2$ with $\rho_\infty$ in the formulae above.

The pseudo-metric $\rho$ is said to be {\em dynamically generating} if for any $x,y \in X$ with $x\ne y$, $\rho(sx,sy) >0$ for some $s \in \Gamma$.
\begin{theorem}\label{thm:entropy}
If $\rho$ is any dynamically generating continuous pseudo-metric on $X$ then
$$h_{\Sigma,\lambda}(X,\Gamma)=h_{\Sigma,\lambda,2}(\rho) = h_{\Sigma,\lambda,\infty}(\rho),$$
$$h_{\Sigma}(X,\Gamma) = h_{\Sigma,2}(\rho) = h_{\Sigma,\infty}(\rho).$$
\end{theorem}
\begin{proof}
This follows from Propositions 2.4 and 3.4 of \cite{KL11b}.
\end{proof}

\section{Entropy of the Harmonic Model}

In this section we prove Theorem~\ref{thm:main1}. Throughout this section we let $\Gamma$ be a countably infinite group, $\Sigma=\{\Gamma_n\}^\infty_{n=1}$ be a sequence of finite-index normal subgroups of $\Gamma$ satisfying $\bigcap_{n=1}^\infty \bigcup_{i\ge n} \Gamma_i = \{e\}$, and $f \in \Z\Gamma$ be well-balanced.

\subsection{The lower bound}

Note that $\Fix_{\Gamma_n}(X_f) \subset X_f$ is a closed $\Gamma$-invariant subgroup. Let $\lambda_n$ be its Haar probability measure.

\begin{lemma} \label{L-measure convergence}
The measure $\lambda_n$ converges in the weak* topology to the Haar probability measure $\lambda$ on $X_f$ as $n\to \infty$.
\end{lemma}

\begin{proof}
For $x\in X_f$, let $A_x:X_f \to X_f$ by the addition map $A_x(y)=x+y$. Each $A_x$ induces a map $(A_x)_*$ on the space $M(X_f)$ of all Borel probability measures on $X_f$. The map $X_f \times M(X_f) \to M(X_f)$ defined by $(x,\mu) \mapsto (A_x)_*\mu$ is continuous (with respect to the weak* topology on $M(X_f)$).

Choose an increasing sequence $\{n_i\}$ of natural numbers so that $\lim_{i\to\infty} \lambda_{n_i} = \lambda_\infty\in M(X_f)$ exists (this is possible by the Banach-Alaoglu Theorem). By the above if $x_i \in \Fix_{\Gamma_{n_i}}(X_f)$ and $\lim_{i\to\infty} x_i = x$ then
$$\lim_{i\to\infty} (A_{x_i})_*\lambda_{n_i} = (A_x)_*\lambda_\infty.$$
Since $\lambda_{n_i}$ is the Haar probability measure on $\Fix_{\Gamma_{n_i}}(X_f)$, $(A_{x_i})_*\lambda_{n_i} = \lambda_{n_i}$, so the above implies $(A_x)_*\lambda_\infty = \lambda_\infty$. Because $\Fix_{\Gamma_{n_i}}(X_f)$ converges in the Hausdorff topology to $X_f$ by  Theorem~\ref{T-dense periodic points}, we have that $(A_x)_*\lambda_\infty = \lambda_\infty$ for every $x\in X_f$ which, by uniqueness of the Haar probability measure, implies that $\lambda_\infty=\lambda$ as required.
\end{proof}

\begin{definition}\label{defn:rho}
For $t \in \R/\Z$ let $t' \in [-1/2,1/2)$ be such that $t' + \Z = t$ and define $|t|: = |t'|$. Similarly, for $x\in (\R/\Z)^\Gamma$, let $x' \in \R^\Gamma$ be the unique element satisfying $x'_s \in [-1/2,1/2)$ and $x'_s + \Z = x_s$ for all $s\in \Gamma$. Define $\|x\|_\infty = \|x'\|_\infty$. Let $\rho$ be the continuous pseudo-metric on $X_f$ defined by $\rho(x,y)=|(x-y)'_e|$. It is easy to check that $\rho$ is dynamically generating.
\end{definition}

For $x\in \Fix_{\Gamma_n}(X_f)$, let $\phi_x:\Gamma/\Gamma_n \to X_f$ be the map defined by $\phi_x(s\Gamma_n) = sx$ for all $s\in \Gamma$. Let $W \subset \Gamma, L \subset C(X_f)$ be non-empty finite sets and $\delta>0$. Note that $\phi_x\in \Map(W,\delta,\Gamma_n)$ for all $x\in \Fix_{\Gamma_n}(X_f)$.
Let $\BAD(W,L,\delta,\Gamma_n)$ be the set of all $x\in \Fix_{\Gamma_n}(X_f)$ such that $\phi_x \notin \Map_\lambda(W,L,\delta,\Gamma_n)$. Because $\phi_x\circ s = s \circ \phi_x$ for every $s\in \Gamma$, $\phi_x \notin \Map_\lambda(W,L,\delta,\Gamma_n)$ if and only if there exists $p\in L$ such that
$$\left| (\phi_x)_*U_n(p) - \lambda(p)\right| > \delta.$$

\begin{lemma}\label{lem:BAD1}
Assume that $\Gamma$ is not virtually $\Zb$ or $\Zb^2$.
Then
$$\lim_{n\to\infty} \lambda_n(\BAD(W,L,\delta,\Gamma_n)) = 0.$$
\end{lemma}
\begin{proof}
The proof is similar to the proof of  \cite[Theorem 3.1]{Bo11}. Let $n\in \Nb$. For each $\sigma \in \{-1,0,1\}^L$, let $\BAD_\sigma(W,L,\delta,\Gamma_n)$ be the set of all $x \in \BAD(W,L,\delta,\Gamma_n)$ such that for every $p \in L$, if $\sigma(p)\ne 0$ then
$$\sigma(p)\left[(\phi_x)_*U_n(p)-\lambda(p)\right] =\sigma(p) \left[ -\lambda(p) + [\Gamma:\Gamma_n]^{-1}\sum_{s\Gamma_n\in \Gamma/\Gamma_n} p(sx) \right] > \delta$$
and if $\sigma(p)=0$ then $\left|(\phi_x)_*U_n(p)-\lambda(p)\right| \le \delta$.
Observe that for each $\sigma \in \{-1,0,1\}^L$, $\BAD_\sigma(W,L,\delta,\Gamma_n)$ is  $\Gamma$-invariant.  Moreover,  $\{\BAD_\sigma(W,L,\delta,\Gamma_n):~\sigma \in \{-1,0,1\}^L\}$ is a partition of $\BAD(W,L,\delta,\Gamma_n)$.

Let $t_{n,\sigma} = \lambda_n(\BAD_\sigma(W,L,\delta,\Gamma_n))$ and $t_{n,G} = 1- \lambda_n(\BAD(W,L,\delta,\Gamma_n))$. So $t_{n,G} + \sum_\sigma t_{n,\sigma} = 1$. For each $\sigma \in \{-1,0,1\}^L$, define a Borel probability measure $\lambda_{n,\sigma}$ on $X_f$ by
$$\lambda_{n,\sigma}(E) = \lambda_n(E \cap \BAD_\sigma(W,L,\delta,\Gamma_n)) t_{n,\sigma}^{-1},\quad \forall \mbox{ Borel } E \subset X_f$$
if $t_{n,\sigma} \ne 0$. Otherwise, define $\lambda_{n,\sigma}$ arbitrarily. Let $\lambda_{n,G}$ be the Borel probability measure on $\Fix_{\Gamma_n}(X_f)$ defined by
$$\lambda_{n,G}(E) = \lambda_n( E \setminus \BAD(W,L,\delta,\Gamma_n) ) t_{n,G}^{-1},\quad \forall \mbox{ Borel } E \subset  X_f$$
if $t_{n,G}\ne 0$. Otherwise, define $t_{n,G}$ arbitrarily. Observe that
$$\lambda_n = t_{n,G}\lambda_{n,G} + \sum_{\sigma} t_{n,\sigma} \lambda_{n,\sigma}.$$
Because the space of Borel probability measures on $X_f$ is weak* sequentially compact (by the Banach-Alaoglu Theorem), there is a subsequence $\{n_i\}_{i=1}^\infty$ such that
\begin{itemize}
\item $\lambda_{n_i,G}$ converges in the weak* topology as $i\to\infty$ to a Borel probability measure $\lambda_{\infty,G}$ on $X_f$,
\item each $\lambda_{n_i,\sigma}$ converges in the weak* topology as $i\to\infty$ to a Borel probability measure $\lambda_{\infty,\sigma}$ on $X_f$,
\item the limits $\lim_{i\to\infty} t_{n_i,G} = t_{\infty,G}$ and $\lim_{i\to\infty} t_{n_i,\sigma} = t_{\infty,\sigma}$ exist for all $\sigma$.
\end{itemize}
By the previous lemma, $\lambda_n$ converges to $\lambda$ as $n\to\infty$. Therefore,
$$\lambda = t_{\infty,G}\lambda_{\infty,G} + \sum_{\sigma} t_{\infty,\sigma} \lambda_{\infty,\sigma}.$$
Because weak* convergence preserves invariance, $\lambda_{\infty,G}$ and each of $\lambda_{\infty,\sigma}$ are $\Gamma$-invariant Borel probability measures on $X_f$. Because $\lambda$ is ergodic by Theorem~\ref{T-dense homoclinic}, this implies that for each $\sigma \in \{-1,0,1\}^L$ with $t_{\infty,\sigma} \ne 0$, $\lambda_{\infty,\sigma} = \lambda$. However, for any $p \in L$ with $\sigma(p)\ne 0$,
$$\sigma(p)\left(\lambda_{\infty,\sigma}(p) - \lambda(p)\right)=\lim_{i\to\infty} \sigma(p)\left(\lambda_{n_i,\sigma}(p) -\lambda(p) \right) \ge \delta.$$
This contradiction implies $t_{\infty,\sigma} = 0$ for all $\sigma \in \{-1,0,1\}^L$ (if $\sigma$ is constantly $0$, then $\BAD_\sigma(W,L,\delta,\Gamma_n)$ is empty so $t_{\infty,\sigma}=0$). Thus $\lim_{n\to\infty} t_{n,\sigma} = 0$ for all $\sigma$ which, since
$$\lambda_n(\BAD(W,L,\delta,\Gamma_n)) = \sum_{\sigma} \lambda_n(\BAD_\sigma(W,L,\delta,\Gamma_n) ) = \sum_{\sigma} t_{n,\sigma},$$
implies the lemma.
\end{proof}

\begin{lemma}\label{lem:well-known}
Let $x \in \ell^\infty(\Gamma)$ and suppose $xf = 0$. Suppose also that for some finite-index subgroup $\Gamma'<\Gamma$, $sx=x$ for all $s\in \Gamma'$. Then $x$ is constant.
\end{lemma}

\begin{proof}
Because $x$ is fixed by a finite-index subgroup, there is an element $s_0\in \Gamma$ such that $x_{s_0} = \min_{s\in \Gamma} x_s$. Because $xf=0$ and $f$ is well-balanced this implies that $x_{s_0t}=x_{s_0}$ for every $t$ in the support of $f$. By induction, $x_{s_0 t}=x_{s_0}$ for every $t$ in the semi-group generated by the support of $f$. By hypothesis, this semi-group is all of $\Gamma$.
\end{proof}

\begin{lemma} \label{L-small suprenorm to constant}
There is a number $C>0$ such that if $x\in \Fix_{\Gamma_n}(X_f)$ for some $n\in \Nb$ satisfies $\|x\|_\infty< C$ then $x$ is constant.
\end{lemma}

\begin{proof}
Let $x'$ be as in Definition~\ref{defn:rho}.
Because $f$ has finite support, there is some number $C>0$ such that if $\|x'\|_\infty=\|x\|_\infty < C$ then $\| x'f\|_\infty < 1$. Since $x\in X_f$, $x'f \in \ell^\infty(\Gamma,\Z)$. So $\|x'f\|_\infty < 1$ implies $x'f = 0$. Because $x'$ is fixed by a finite-index subgroup the previous lemma implies $x'$ is constant and hence $x$ is constant.
\end{proof}

\begin{lemma}\label{lem:lower}
Assume that $\Gamma$ is not virtually $\Zb$ or $\Zb^2$.
Then
$$h_{\Sigma,\lambda}(X_f,\Gamma) \ge \limsup_{n\to \infty} [\Gamma:\Gamma_n]^{-1} \log |\Fix_{\Gamma_n}(X_f)|.$$
\end{lemma}

\begin{proof}
Let $W \subset \Gamma, L \subset C(X_f)$ be non-empty finite sets and $\delta>0$.
Let us identify $\R/\Z$ with the constant functions (on $\Gamma$) in $X_f$.
It follows easily from Lemma~\ref{L-small suprenorm to constant} that the connected component of $\Fix_{\Gamma_n}(X_f)$ containing the identity element is exactly $\Rb/\Zb$.


Choose a maximal set $Y_n \subset \Fix_{\Gamma_n}(X_f)$ such that $Y_n \cap \BAD(W,L, \delta, \Gamma_n) = \emptyset$ and for each $x \in Y_n$ and $t \in \R/\Z$ with $t\ne 0$, $x+t \notin Y_n$. By Lemma~\ref{lem:BAD1}, $\lim_{n\to \infty} |Y_n|^{-1} |\Fix_{\Gamma_n}(X_f)| = 1$. Let $C>0$ be the constant in Lemma~\ref{L-small suprenorm to constant}. By Lemma~\ref{L-small suprenorm to constant} if $x\ne y \in Y_n$ then $\|x-y\|_\infty \ge C$ which implies $\rho_\infty(\phi_x,\phi_y)\ge C$.

Therefore, if $0<\epsilon<C$ then $\{\phi_y:~y\in Y_n\}$ is $\epsilon$-separated with respect to $\rho_\infty$ which implies
$$ N_\epsilon( \Map_\lambda(W,L,\delta,\Gamma_n), \rho_\infty) \ge |Y_n|.$$
Because $\lim_{n\to \infty} |Y_n|^{-1} |\Fix_{\Gamma_n}(X_f)| = 1$, this implies
$$\limsup_{n\to \infty} [\Gamma:\Gamma_n]^{-1} \log |\Fix_{\Gamma_n}(X_f)| \le h_{\Sigma,\lambda}(X_f,\Gamma).$$
\end{proof}

\begin{lemma} \label{L-amenable lower bound}
Assume that $\Gamma$ is amenable. Then
$$h_{\Sigma,\lambda}(X_f,\Gamma) \ge \limsup_{n\to \infty} [\Gamma:\Gamma_n]^{-1} \log |\Fix_{\Gamma_n}(X_f)|.$$
\end{lemma}
\begin{proof} The argument in the proof of Lemma~\ref{lem:lower} shows that
$$ h_{\Sigma}(X_f,\Gamma) \ge \limsup_{n\to \infty} [\Gamma:\Gamma_n]^{-1} \log |\Fix_{\Gamma_n}(X_f)|.$$
Note that $h_{\Sigma,\lambda}(X_f,\Gamma)$ coincides with the classical measure-theoretic entropy  $h_\lambda(X_f, \Gamma)$ \cite[Theorem 1.2]{Bo12}  \cite[Theorem 6.7]{KL11b}, and $h_{\Sigma}(X_f,\Gamma)$ coincides with the classical topological entropy  $h(X_f, \Gamma)$  \cite[Theorem 5.3]{KL11b}.
Since $\Gamma$ acts on $X_f$ by continuous group automorphism and $\lambda$ is the Haar probability measure of $X_f$, one has $h_\lambda(X_f, \Gamma)=h(X_f, \Gamma)$ \cite{Berg, De06}. Therefore
$$ h_{\Sigma,\lambda}(X_f,\Gamma)=h_\lambda(X_f, \Gamma)=h(X_f, \Gamma)=h_{\Sigma}(X_f,\Gamma)  \ge \limsup_{n\to \infty} [\Gamma:\Gamma_n]^{-1} \log |\Fix_{\Gamma_n}(X_f)|.$$
\end{proof}

\begin{lemma} \label{L-lower bound}
We have
$$ h_{\Sigma,\lambda}(X_f,\Gamma) \ge \limsup_{n\to \infty} [\Gamma:\Gamma_n]^{-1} \log |\Fix_{\Gamma_n}(X_f)|.$$
\end{lemma}
\begin{proof} If $\Gamma$ has a finite-index normal subgroup isomorphic to  $\Zb$ or  $\Zb^2$, then $\Gamma$ is amenable \cite[Theorem G.2.1 and Proposition G.2.2]{BHV}. Thus the assertion follows from Lemmas~\ref{lem:lower} and \ref{L-amenable lower bound}.
\end{proof}

\subsection{The upper bound}



Let $U_n$ be the uniform probability measure on $\Gamma/\Gamma_n$. For $x,y\in \R^{\Gamma/\Gamma_n}$, let $\langle x,y\rangle_U$ be the inner product with respect to $U_n$ and $\|x\|_{p,U} := \left([\Gamma: \Gamma_n]^{-1} \sum_{s\Gamma_n \in \Gamma/\Gamma_n} |x_{s\Gamma_n}|^p\right)^{1/p}$ for $p\ge 1$.

For $x\in (\R/\Z)^{\Gamma/\Gamma_n}$, let $x'$ be as in Definition~\ref{defn:rho}.
Let $|x|=|x'|$ and  $\|x\|_{p,U} = \|x'\|_{p,U}$ for $p\ge 1$.
 We will use $\langle \cdot ,\cdot \rangle$ and $\|\cdot\|_p$ to denote the inner product and $\ell^p$-norm with respect to the counting measure.

\begin{lemma} \label{L-diffrent norms}
For any $x\in (\R/\Z)^{\Gamma/\Gamma_n}$,
$$\|x\|_{2,U}^2 \le \|x \|_{1,U} \le  \|x\|_{2,U}\le 1/2.$$
\end{lemma}

\begin{proof}
This first inequality is immediate from $\|x\|_{\infty} \le 1/2$. Note
\begin{eqnarray*}
\|x\|_1^2 &=& \sum_{s\Gamma_n \in \Gamma/\Gamma_n} \sum_{t\Gamma_n \in \Gamma/\Gamma_n}  |x_{s\Gamma_n}|  |x_{t\Gamma_n}| = \sum_{s\Gamma_n \in \Gamma/\Gamma_n} \sum_{t\Gamma_n \in \Gamma/\Gamma_n}  |x_{s\Gamma_n}|  |x_{ts\Gamma_n}| \\
&=& \sum_{t\Gamma_n \in \Gamma/\Gamma_n} \langle |x|, |x \circ t\Gamma_n|\rangle.
\end{eqnarray*}
By the Cauchy-Schwarz inequality, for any $t\in \Gamma$,
$$\langle |x|, |x \circ t\Gamma_n|\rangle \le \|x\|_2 \|x \circ t\Gamma_n\|_2 = \|x\|_2^2.$$
Hence
$$\|x\|_1^2 \le [\Gamma:\Gamma_n] \|x\|_2^2.$$
Since $\|x\|_{1,U} = [\Gamma:\Gamma_n]^{-1}\|x\|_1$ and $\|x\|_{2,U}^2 =[\Gamma:\Gamma_n]^{-1}  \|x\|_2^2$, this implies the second inequality. The last one follows from $\|x\|_{2,U} \le \|x\|_\infty \le 1/2$.
\end{proof}

Recall from Section~\ref{S-notation} that for a countable group $\Gamma'$ and $g\in \Rb\Gamma'$,
if $g$ is positive in $\cN\Gamma'$, then we have the spectral measure of $g$ on
$[0, \|R_g\|]\subset [0, \|g\|_1]$ determined by \eqref{E-spectral measure}.
For each $n\in \Nb$ we denote by $\pi_n$ the natural algebra homomorphism $\Rb\Gamma\rightarrow \Rb(\Gamma/\Gamma_n)$.

\begin{lemma} \label{L-bound preimage}
Let $g\in \Rb\Gamma$ such that
the kernel of $g$ on $\ell^2(\Gamma, \Cb)$ is $\{0\}$,
and $\pi_n(g)$ is positive in $\cN(\Gamma/\Gamma_n)$ for all $n\in \Nb$.
For each $n\in \Nb$ and $\eta>0$ denote by $B_{n, \eta}$ the set of $x\in \Rb(\Gamma/\Gamma_n)$ satisfying
$\|x\pi_n(g)\|_{2, U}\le \eta$ and $\|P_n(x)\|_{2, U}\le 1$, where $P_n$ denotes the orthogonal projection from $\ell^2(\Gamma/\Gamma_n, \Cb)$ onto $\ker \pi_n(g)$.
For each $n\in \Nb$ denote by $\mu_n$ the spectral measure of $\pi_n(g)$ on $[0, \|g\|_1]$.
Let $\zeta>1$, $1>\varepsilon>0$, and $1/2>\kappa>0$. Then
there exists $\eta>0$ such that when $n\in \Nb$ is large enough, one has
$$N_\varepsilon(B_{n, \eta}, \|\cdot\|_{2, U})<\zeta^{[\Gamma: \Gamma_n]}\exp(-[\Gamma: \Gamma_n]\int_{0+}^{\kappa}\log t \, d\mu_n(t))$$
where $N_\varepsilon(\cdot,\cdot)$ is as in Definition \ref{defn:separating}.
\end{lemma}
\begin{proof}
Since $\pi_n(g)$ is positive in $\cN(\Gamma/\Gamma_n)$, one has $(\pi_n(g))^*=\pi_n(g)$.

Let $Y_n$ be a maximal $(\|\cdot\|_{2, U},\varepsilon/6)$-separated subset of the closed unit ball of $\ker \pi_n(g)$ under $\|\cdot \|_{2, U}$. Then
the open $\varepsilon/12$-balls centered at $y$ under $\|\cdot\|_{2, U}$ for all $y\in Y_n$ are pairwise disjoint, and their union is  contained in the open $2$-ball of $\ker \pi_n(g)$ under $\|\cdot \|_{2, U}$. Comparing the volumes we obtain $|Y_n|\le (24/\varepsilon)^{\dim_{\Rb} \ker \pi_n(g)}$.

For each $n\in \Nb$
denote by $V_{n, \kappa}$ the linear span of the eigenvectors of $\pi_n(g)$ $\ell^2(\Gamma/\Gamma_n, \Cb)$ with eigenvalue no bigger than $\kappa$.
Note that $V_{n, 0}=\ker \pi_n(g)$.
Since $\ker (g^*g)=\ker(g)=0$, by a result of L\"{u}ck \cite[Theorem 2.3]{Luck94}
(it was assumed in \cite{Luck94} that $\Gamma_n\supset \Gamma_{n+1}$ for all $n\in \Nb$; but the argument there holds in general), one has
$\lim_{n\to \infty}[\Gamma: \Gamma_n]^{-1}\dim_{\Cb} \ker \pi_n(g)=\lim_{n\to \infty}[\Gamma: \Gamma_n]^{-1}\dim_{\Cb} \ker \pi_n(g^*g)=0$. It follows that $|Y_n|\le \zeta^{[\Gamma:\Gamma_n]}$ when $n$ is large enough.

Denote by $P_{n, \kappa}$ the orthogonal projection of $\ell^2(\Gamma/\Gamma_n, \Cb)$ onto $V_{n, \kappa}$. Set $\eta=\min(\varepsilon/24, \kappa \varepsilon /12)$.
Note that for each $x\in \Rb(\Gamma/\Gamma_n)$ one has
$$\|x\pi_n(g)\|_{2, U}^2=\|(P_{n, \kappa}(x))\pi_n(g)\|_{2, U}^2+\|(x-P_{n, \kappa}(x))\pi_n(g)\|_{2, U}^2\ge \kappa^2\|x-P_{n, \kappa}(x)\|_{2, U}^2.$$
Thus $\|x-P_{n, \kappa}(x)\|_{2, U}\le \eta/\kappa\le \varepsilon/12$ for every $x\in B_{n, \eta}$.
 Then every two points in $({\rm Id}-P_{n, \kappa})(B_{n, \eta})$ have $\|\cdot \|_{2, U}$-distance at most $\varepsilon/6$. Let $X_n$ be a one-point subset of $({\rm Id}-P_{n, \kappa})(B_{n, \eta})$. Then $X_n$ is a maximal $(\varepsilon/6)$-separated subset of $({\rm Id}-P_{n, \kappa})(B_{n, \eta})$ under $\|\cdot \|_{2, U}$.

Denote by $E_{n, \kappa}$ the ordered set of all eigenvalues of $\pi_n(g)$ in $(0, \kappa]$ listed with multiplicity.
Let $Z_n$ be a maximal $(\varepsilon/6)$-separated subset of $(P_{n, \kappa}-P_{n, 0})(B_{n, \eta})$ under $\|\cdot \|_{2, U}$.
For each $z\in Z_n$ denote by $B_z$ the open ball centered at $z$ with radius $\varepsilon/12$ under $\|\cdot \|_{2, U}$. Note that $\|x\pi_n(g)\|_{2, U}\le \kappa\|x\|_{2, U}$ for all $x\in V_{n, \kappa}\ominus V_{n, 0}$. Thus every element in $(\bigcup_{z\in Z_n}B_z)\pi_n(g)$ has $\|\cdot \|_{2, U}$-norm at most $\eta+\kappa \varepsilon/12$.  The volume of $(\bigcup_{z\in Z_n}B_z)\pi_n(g)$ is $\det(\pi_n(g)|_{V_{n, \kappa}\ominus V_{n, 0}})=\prod_{t\in E_{n, \kappa}}t$ times the volume of $\bigcup_{z\in Z_n}B_z$. It follows that
$$ |Z_n| \prod_{t\in E_{n, \kappa}} t\le \left(\frac{\eta+\kappa \varepsilon/12}{\varepsilon/12}\right)^{\dim_\Rb (V_{n, \kappa}\ominus V_{n, 0})}=\left(\frac{12\eta+\kappa \varepsilon}{\varepsilon}\right)^{\dim_\Rb (V_{n, \kappa}\ominus V_{n, 0})}\le 1.$$

Note that for every $t\in [0, \|g\|_1]$,  the measure $\mu_n(\{t\})$ is exactly $[\Gamma:\Gamma_n]^{-1}$ times the multiplicity of $t$ as an eigenvalue of $\pi_n(g)$.
When $n\in \Nb$ is sufficiently large, we have
\begin{align*}
N_{\varepsilon}(B_{n, \eta}, \|\cdot \|_{2, U})&\le |X_n|\cdot |Y_n|\cdot |Z_n| \le \zeta^{[\Gamma: \Gamma_n]}\prod_{t\in E_{n, \kappa}}t^{-1} \\
&=\zeta^{[\Gamma: \Gamma_n]}\exp\left(-[\Gamma: \Gamma_n]\int_{0+}^{\kappa}\log t \, d\mu_n(t)\right)
\end{align*}
as desired.
\end{proof}

\begin{lemma} \label{L-measure weak convergence}
Let $g\in \Rb\Gamma$ be such that
$g$ is positive in $\cN\Gamma$, and $\pi_n(g)$ is positive in $\cN(\Gamma/\Gamma_n)$ for all $n\in \Nb$.
Denote by $\mu$ the spectral measure of $g$ on $[0, \|g\|_1]$.
For each $n\in \Nb$ denote by $\mu_n$ the spectral measure of $\pi_n(g)$ on $[0, \|g\|_1]$.
Let $\min(1, \|g\|_1)>\kappa>0$. Then
$$ \limsup_{n\to \infty}\int_{\kappa+}^{\|g\|_1}\log t \, d\mu_n(t)\le \int_{\kappa+}^{\|g\|_1}\log t \, d\mu(t).$$
\end{lemma}
\begin{proof} It suffices to show
$$ \limsup_{n\to \infty}\int_{\kappa+}^{\|g\|_1}\log t \, d\mu_n(t)\le \eta(1+\|g\|_1)+\int_{\kappa+}^{\|g\|_1}\log t \, d\mu(t).$$
for every $\eta>0$. Let $\eta>0$.

For each sufficiently large $k\in \Nb$ define a real-valued continuous function $q_k$ on $[0, \|g\|_1]$ to be $0$ on $[0, \kappa]$,
$\log t$ at $t\in [\kappa+1/k, \|g\|_1]$, and linear on $[\kappa, \kappa+1/k]$. By the Lebesgue dominated convergence theorem one has $\int_{\kappa+}^{\|g\|_1} q_k(t)\, d\mu(t)\to \int_{\kappa+}^{\|g\|_1}\log t\, d\mu(t)$ as $k\to \infty$. Fix $k\in \Nb$ with $\int_{\kappa+}^{\|g\|_1} q_k(t)\, d\mu(t)\le \int_{\kappa+}^{\|g\|_1}\log t\, d\mu(t)+\eta$, and take
 a real-coefficients polynomial $p$ such that $q_k+\eta \ge p\ge q_k$ on $[0, \|g\|_1]$. Then $p(t)\ge q_k(t)\ge \log t$ for all $t\in (0, \|g\|_1]$, and
\begin{align*}
\tr_{\cN\Gamma}(p(g))&\overset{\eqref{E-spectral measure}}= \int_{0}^{\|g\|_1}p(t) \, d\mu(t) \le \eta \|g\|_1 +\int_{0}^{\|g\|_1}q_k(t) \, d\mu(t) \\
&= \eta \|g\|_1 +\int_{\kappa+}^{\|g\|_1}q_k(t) \, d\mu(t) \le \eta(1+\|g\|_1)+\int_{\kappa+}^{\|g\|_1}\log t \, d\mu(t).
\end{align*}

When $n\in \Nb$ is large enough, one has $\tr_{\cN(\Gamma/\Gamma_n)}(p(\pi_n(g)))=\tr_{\cN\Gamma}(p(g))$ \cite[Lemma 2.6]{Luck94}, whence
$$ \tr_{\cN\Gamma}(p(g))=\tr_{\cN(\Gamma/\Gamma_n)}(p(\pi_n(g)))\overset{\eqref{E-spectral measure}}=\int_{0}^{\|g\|_1}p(t) \, d\mu_n(t)\ge  \int_{\kappa+}^{\|g\|_1}p(t) \, d\mu_n(t)\ge \int_{\kappa+}^{\|g\|_1}\log t \, d\mu_n(t) .$$
Thus
\begin{align*}
\limsup_{n\to \infty}\int_{\kappa+}^{\|g\|_1}\log t \, d\mu_n(t)\le \tr_{\cN\Gamma}(p(g))\le \eta(1+\|g\|_1)+\int_{\kappa+}^{\|g\|_1}\log t \, d\mu(t)
\end{align*}
as desired.
\end{proof}

\begin{lemma} \label{L-upper bound}
We have
$$h_{\Sigma}(X_f,\Gamma) \le \log \ddet_{\cN\Gamma} f.$$
\end{lemma}
\begin{proof} Let $\rho$ be the pseudo-metric on $X_f$ defined as in Definition \ref{defn:rho}. Let $\zeta>1$ and $1/2>\kappa>0$. Let $1>\varepsilon>0$.
Denote by $W$ the support of $f$. An argument similar to that in the proof of Lemma~\ref{lem:well-known} shows that the kernel of $f$ on $\ell^2(\Gamma, \Cb)$ is $\{0\}$. Note that $f\ge 0$ in $\cN\Gamma$ and
$\pi_n(f)\ge 0$ in $\cN(\Gamma/\Gamma_n)$ for all $n\in \Nb$.
Take $\eta>0$ in Lemma~\ref{L-bound preimage} for $g=f$. Take $\delta>0$ such that
$2\|f\|_2^{1/2}|W|^{1/4}\delta^{1/2}<\eta$.

For $\phi \in \Map(W,\delta,\Gamma_n)$ define $y_\phi \in (\R/\Z)^{\Gamma/\Gamma_n}$  by $y_\phi(s\Gamma_n) = \phi(s^{-1}\Gamma_n)_e$. Note that
\begin{eqnarray*}
\|y_\phi \pi_n(f)\|_{2,U}^2&=&[\Gamma:\Gamma_n]^{-1} \sum_{s\Gamma_n \in \Gamma/\Gamma_n}  \left| \sum_{t\in W} y_\phi(st\Gamma_n) f(t^{-1})\right|^2 \\
&=& [\Gamma:\Gamma_n]^{-1} \sum_{s\Gamma_n \in \Gamma/\Gamma_n}  \left| \sum_{t\in W} \phi(t^{-1}s^{-1}\Gamma_n)_e f(t^{-1})\right|^2 \\
&=& [\Gamma:\Gamma_n]^{-1} \sum_{s\Gamma_n \in \Gamma/\Gamma_n}  \left| \sum_{t\in W} \left(\phi(t^{-1}s^{-1}\Gamma_n)_e - [t^{-1}\phi(s^{-1}\Gamma_n)]_e + [t^{-1}\phi(s^{-1}\Gamma_n)]_e\right)  f(t^{-1})\right|^2 \\
&=& [\Gamma:\Gamma_n]^{-1} \sum_{s\Gamma_n \in \Gamma/\Gamma_n}  \left| \sum_{t\in W} \left(\phi(t^{-1}s^{-1}\Gamma_n)_e - [t^{-1}\phi(s^{-1}\Gamma_n)]_e \right)  f(t^{-1})\right|^2 \\
&\le& [\Gamma:\Gamma_n]^{-1} \sum_{s\Gamma_n \in \Gamma/\Gamma_n}   \left( \sum_{t\in W} \left|\phi(t^{-1}s^{-1}\Gamma_n)_e - [t^{-1}\phi(s^{-1}\Gamma_n)]_e \right|^2 \right) \|f \|_2^2 \\
&=& \|f\|_2^2 \sum_{t \in W} \rho_2(\phi \circ t^{-1}, t^{-1} \circ \phi)^2\\
&\le& \|f\|_2^2 |W| \delta^2.
\end{eqnarray*}

For each $\phi\in \Map(W,\delta,\Gamma_n)$ take $\tilde{y}_\phi \in [-1/2, 1/2)^{\Gamma/\Gamma_n}$ such that $\tilde{y}_\phi + \Z^{\Gamma/\Gamma_n} =y_\phi$. Then there exists $z_\phi\in \Z^{\Gamma/\Gamma_n}$ such that $\tilde{y}_\phi\pi_n(f) - z_\phi \in [-1/2,1/2)^{\Gamma/\Gamma_n}$ which implies $\|\tilde{y}_\phi \pi_n(f)-z_\phi\|_{2,U} = \|y_\phi \pi_n(f)\|_{2,U}\le \|f\|_2 |W|^{1/2} \delta$.

Let $S_n:\R^{\Gamma/\Gamma_n} \to \R$ be the sum function: $S_n(y) = \sum_{s\Gamma_n \in \Gamma/\Gamma_n} y_{s\Gamma_n}$.  Note that
$$|S_n(\tilde{y}_\phi \pi_n(f)-z_\phi)| \le \|y_\phi \pi_n(f)\|_1 \le (1/2) [\Gamma:\Gamma_n].$$
Note that $S_n(y\pi_n(f))=0$ for every $y \in \R^{\Gamma/\Gamma_n}$. In particular, $S_n(\tilde{y}_\phi\pi_n(f)-z_\phi)=-S_n(z_\phi) \in \Z$.
So there exists $z'_\phi \in \{-1,0,1\}^{\Gamma/\Gamma_n}$ such that $S_n(\tilde{y}_\phi\pi_n(f)-z_\phi-z'_\phi)=0$ and $\|z'_\phi\|_1 \le \|y_\phi \pi_n(f)\|_1$. So by Lemma~\ref{L-diffrent norms}
\begin{eqnarray*}
\| \tilde{y}_\phi\pi_n(f)-z_\phi-z'_\phi\|_{2,U} &\le& \|\tilde{y}_\phi \pi_n(f) - z_\phi\|_{2,U} + \|z'_\phi\|_{2,U} \\
&\le& \|\tilde{y}_\phi \pi_n(f) - z_\phi\|_{2,U} + \|z'_\phi\|_{1,U}^{1/2}\\
 &\le& \|y_\phi \pi_n(f)\|_{2,U} +\|y_\phi \pi_n(f)\|_{2,U}^{1/2} \\
 &\le& (1+2^{-1/2})\|y_\phi \pi_n(f)\|_{2,U}^{1/2} \\
 &\le& 2\|f\|_2^{1/2}|W|^{1/4}\delta^{1/2}<\eta.
 \end{eqnarray*}

 Note that $\ker \pi_n(f)$ is the constants in $\ell^2(\Gamma/\Gamma_n, \Cb)$. Denote by $\ell^2_0(\Gamma/\Gamma_n, \Cb)$ the orthogonal complement of
 the constants in $ \ell^2(\Gamma/\Gamma_n, \Cb)$, and set $\ell^2_0(\Gamma/\Gamma_n, \Rb)=\ell^2(\Gamma/\Gamma_n, \Rb)\cap \ell^2_0(\Gamma/\Gamma_n, \Cb)$.
 Note that $y\in \R^{\Gamma/\Gamma_n}$ is in $\ell^2_0(\Gamma/\Gamma_n, \Rb)$ exactly when $S_n(y)=0$.
The operator $\pi_n(f)$ is invertible as an operator from $\ell_0^2(\Gamma/\Gamma_n, \Cb)$ to itself.
Since $\pi_n(f)$  preserves $\ell^2(\Gamma/\Gamma_n, \Rb)$, it is also invertible from $\ell_0^2(\Gamma/\Gamma_n, \Rb)$ to itself.
Therefore there exists $\tilde{x}_\phi\in \ell_0^2(\Gamma/\Gamma_n, \Rb)$ such that $\tilde{x}_\phi\pi_n(f)=z_\phi+z'_\phi$.
Let $P_n$ and $B_{n, \eta}$ be as in Lemma~\ref{L-bound preimage} for $g=f$.
Then
$$ \|(\tilde{y}_\phi-\tilde{x}_\phi)\pi_n(f)\|_{2, U}=\| \tilde{y}_\phi\pi_n(f)-z_\phi-z'_\phi\|_{2,U}<\eta.$$
Note that $\|P_n(\tilde{y}_\phi-\tilde{x}_\phi)\|_{2, U}=\|P_n(\tilde{y}_\phi)\|_{2, U}\le \|\tilde{y}_\phi\|_{2, U}\le 1/2$.
Therefore $\tilde{y}_\phi-\tilde{x}_\phi\in B_{n, \eta}$.

Let $\Phi_n$ be a  $(\rho_2, 2\varepsilon)$-separated subset of $\Map(W,\delta,\Gamma_n)$ with $|\Phi_n|=N_{2\varepsilon}(\Map(W,\delta,\Gamma_n), \rho_2)$.

Let $\phi \in \Phi_n$. Denote by $B_\phi$ the set of all $\psi\in \Phi_n$ satisfying
$\|(\tilde{y}_\phi-\tilde{x}_\phi)-(\tilde{y}_\psi-\tilde{x}_\psi)\|_{2, U}<\varepsilon$.
Denote $\Zb(\Gamma/\Gamma_n)\cap \ell_0^2(\Gamma/\Gamma_n, \Rb)$ by $\Zb_0(\Gamma/\Gamma_n)$.
We claim that the map $B_\phi\rightarrow \Zb_0(\Gamma/\Gamma_n)/\Zb_0(\Gamma/\Gamma_n)\pi_n(f)$ sending
$\psi$ to $z_\psi+z'_\psi+\Zb_0(\Gamma/\Gamma_n)\pi_n(f)$ is injective.
Let $\psi, \varphi\in B_\phi$. Then
\begin{align*}
\|(\tilde{y}_\psi-\tilde{x}_\psi)-(\tilde{y}_\varphi-\tilde{x}_\varphi)\|_{2, U}<2\varepsilon.
\end{align*}
Suppose that $z_\psi+z'_\psi+\Zb_0(\Gamma/\Gamma_n)\pi_n(f)=z_\varphi+z'_\varphi+\Zb_0(\Gamma/\Gamma_n)\pi_n(f)$.
Then $\tilde{x}_\psi\pi_n(f)=\tilde{x}_\varphi\pi_n(f)+w\pi_n(f)$ for some $w\in \Zb_0(\Gamma/\Gamma_n)$. Since the right multiplication by $\pi_n(f)$ is injective on $\ell_0^2(\Gamma/\Gamma_n, \Rb)$, we get
$\tilde{x}_\psi=\tilde{x}_\varphi+w$, which implies that
$$\rho_2(\psi, \varphi)=\|y_\psi-y_\varphi\|_{2, U}\le \|(\tilde{y}_\psi-\tilde{x}_\psi)-(\tilde{y}_\varphi-\tilde{x}_\varphi)\|_{2, U}<2\varepsilon,$$
and thus $\psi=\varphi$. This proves our claim.

Therefore $|B_\phi|\le |\Zb_0(\Gamma/\Gamma_n)/\Zb_0(\Gamma/\Gamma_n)\pi_n(f)|$. Note that $\ell_0^2(\Gamma/\Gamma_n, \Rb)$ is the linear span of $\Zb_0(\Gamma/\Gamma_n)$. So any basis of $\Zb_0(\Gamma/\Gamma_n)$ as a free abelian group is also a basis for $\ell_0^2(\Gamma/\Gamma_n, \Rb)$ as an $\R$-vector space.
 Thus by Lemma~\ref{lem:correspondence}
one has
$|\Zb_0(\Gamma/\Gamma_n)/\Zb_0(\Gamma/\Gamma_n)\pi_n(f)|=|\det (\pi_n(f)|_{\ell_0^2(\Gamma/\Gamma_n, \Rb)})|$.
Therefore
$$|B_\phi|\le |\det (\pi_n(f)|_{\ell_0^2(\Gamma/\Gamma_n, \Rb)})|.$$

Now we have
\begin{align*}
|\Phi_n|\le N_\varepsilon(B_{n, \eta}, \|\cdot \|_{2, U}) \max_{\phi\in  \Phi_n}|B_\phi|\le N_\varepsilon(B_{n, \eta}, \|\cdot \|_{2, U}) |\det (\pi_n(f)|_{\ell_0^2(\Gamma/\Gamma_n, \Rb)})|.
\end{align*}
Let $\mu$ and $\mu_n$ be as in Lemma~\ref{L-measure weak convergence} for $g=f$.
When $n\in \Nb$ is large enough, by Lemma~\ref{L-bound preimage} we have
$$ N_\varepsilon(B_{n, \eta}, \|\cdot\|_{2, U})<\zeta^{[\Gamma: \Gamma_n]}\exp\left(-[\Gamma: \Gamma_n]\int_{0+}^{\kappa}\log t \, d\mu_n(t)\right),$$
whence
\begin{align*}
N_{2\varepsilon}(\Map(W,\delta,\Gamma_n), \rho_2)
&=|\Phi_n| \\
&\le \zeta^{[\Gamma: \Gamma_n]}\exp\left(-[\Gamma: \Gamma_n]\int_{0+}^{\kappa}\log t \, d\mu_n(t)\right)\left|\det (\pi_n(f)|_{\ell_0^2(\Gamma/\Gamma_n, \Rb)})\right| \\
&=\zeta^{[\Gamma: \Gamma_n]}\exp\left([\Gamma: \Gamma_n]\int_{\kappa+}^{\|f\|_1}\log t \, d\mu_n(t)\right).
\end{align*}
It follows that
\begin{align*}
\limsup_{n\to \infty}\frac{1}{[\Gamma: \Gamma_n]}\log N_{2\varepsilon}(\Map(W,\delta,\Gamma_n), \rho)&\le \log \zeta+\limsup_{n\to \infty}\int_{\kappa+}^{\|f\|_1}\log t \, d\mu_n(t) \\
&\le \log \zeta+\int_{\kappa+}^{\|f\|_1}\log t \, d\mu(t),
\end{align*}
where the second inequality comes from Lemma~\ref{L-measure weak convergence}. Therefore
$$ h_{\Sigma}(X_f,\Gamma) \le \log \zeta+\int_{\kappa+}^{\|f\|_1}\log t \, d\mu(t).$$
Letting $\zeta\to 1+$ and $\kappa\to 0+$, we get
$$ h_{\Sigma}(X_f,\Gamma) \le \int_{0+}^{\|f\|_1}\log t \, d\mu(t)\overset{\eqref{E-determinant}}=\log \ddet_{\cN\Gamma} f.$$
\end{proof}

We are ready to prove Theorem~\ref{thm:main1}.

\begin{proof}[Proof of Theorem~\ref{thm:main1}]
From Lemmas \ref{L-lower bound} and \ref{L-upper bound} and Theorem~\ref{T-det vs number of fixed point} we obtain
$$h_{\Sigma}(X_f,\Gamma) \le \log \ddet_{\cN\Gamma} f=\lim_{n\to \infty} [\Gamma:\Gamma_n]^{-1} \log |\Fix_{\Gamma_n}(X_f)| \le h_{\Sigma,\lambda}(X_f,\Gamma).$$
It follows immediately from Theorem \ref{thm:entropy} that $h_{\Sigma}(X_f,\Gamma) \ge h_{\Sigma,\lambda}(X_f,\Gamma).$
\end{proof}


\section{Entropy of the Wired Spanning Forest}
The purpose of this section is to prove Theorem \ref{thm:WSF}. To begin, let us set notation. Recall that $\Sigma=\{\Gamma_n\}^\infty_{n=1}$ a sequence of finite-index normal subgroups of $\Gamma$ satisfying $\bigcap_{n=1}^\infty \bigcup_{i\ge n} \Gamma_i = \{e\}$. All graphs in this paper are allowed to have multiple edges and loops. Let $f \in \Z\Gamma$ be well-balanced. The Cayley graph $C(\Gamma,f)$ has vertex set $\Gamma$. For each $v \in \Gamma$ and $s \ne e$, there are $|f_s|$ edges from $v$ to $vs$. Similarly, we let $C_n^f=C(\Gamma/\Gamma_n,f)$ be the graph with vertex set $\Gamma/\Gamma_n$ such that for each $g\Gamma_n \in \Gamma/\Gamma_n$ and $s \in \Gamma$ there are $|f_s|$ edges from $g\Gamma_n$ to $gs\Gamma_n$. For the sake of convenience we let $E=E(\Gamma,f)$ denote the edge set of $C(\Gamma,f)$ and $E_n=E_n^f$ denote the edge set of $C_n^f$. Recall the definition of $S$ and $S_*$ from Notation \ref{note:S}.


Let $\pi_n:\Gamma \to \Gamma/\Gamma_n$ denote the quotient map. We also denote by $\pi_n$ the induced map from $\R\Gamma$ to $\R\Gamma/\Gamma_n$ as well as the map from $E(\Gamma,f)$ (the edge set of $C(\Gamma,f)$) to $E_n^f$ (the edge set of $C_n^f$).




\subsection{The lower bound}


If $\cH \subset C_n^f$ is a subgraph then its {\em lift} $\tilde{\cH} \subset C(\Gamma,f)$ is the subgraph which contains an edge $gs$ (for $g\in \Gamma, s\in S_*$) if and only if $\cH$ contains $\pi_n(gs)$. Let $2^E$ be the set of all spanning subgraphs of $C(\Gamma,f)$ and let $2^{E_n}$ be the set of all spanning subgraphs of $C_n^f$. Let $\nu_n$ be the probability measure on $2^{E_n}$ which is uniformly distributed on the collection of spanning trees of $C_n^f$. Let $\tilde{\nu}_n$ be its lift to $2^E$. To be precise, $\tilde{\nu}$ is uniformly distributed on the set of lifts $\tilde{\cT}$ of spanning trees $\cT \in 2^{E_n}$.

\begin{lemma}\label{lem:WSF2}
$\tilde{\nu}_n$ converges in the weak* topology to $\nu_{WSF}$ as $n$ tends to infinity.
\end{lemma}

\begin{remark}
This lemma is a special case of \cite[Proposition 7.1]{AL07}. It is also contained in \cite[Theorem 4.3]{Bo04}. However, there is an error in the proof of  \cite[Theorem 4.3]{Bo04} (namely, the fact that subspaces $S_i$ increase to $l^2_-(\Gamma)$ does not logically imply that $P_{S_i}(\star)$ converges to $\star$ in the strong operator topology). For the reader's convenience we provide another proof based on a negative correlations result of Feder and Mihail.
\end{remark}

Let $\cG=(V^\cG,E^\cG)$ be a finite connected graph. A collection $\cA$ of spanning subgraphs is {\em increasing} if $x \subset y$ and $x\in \cA$ implies $y\in \cA$. We say that $\cA$ {\em ignores} an edge $\fe$ if $x \setminus \{\fe\} = y \setminus \{\fe\}$ and $x \in \cA$ implies $y \in \cA$.

\begin{lemma}
If $\cA$ is increasing, $\cA$ ignores $\fe$, $\fe$ is not a loop and $T$ denotes the uniform spanning tree on $\cG$ then
$${\bf P}(T \in \cA) \ge {\bf P}(T \in \cA| \fe \in T) = \frac{{\bf P}(T \in \cA, \fe \in T) }{{\bf P}(\fe \in T) }$$
 where ${\bf P}(\cdot)$ denotes probability. Equivalently, ${\bf P}(T \in \cA) \le {\bf P}(T \in \cA| \fe \notin T)$ whenever this is well-defined (i.e., whenever ${\bf P}(\fe \notin T)>0$, or equivalently, whenever $\cG\setminus \{\fe\}$ is connected).
\end{lemma}

\begin{proof}
This result is due to Feder and Mihail \cite{FM92}. The first statement is reproduced in \cite[Theorem 4.4]{BLPS01}. To see that the second inequality is equivalent to the first observe that
$$ {\bf P}(T \in \cA| \fe \notin T) = \frac{ {\bf P}(T \in \cA, \fe\notin T) }{{\bf P}(\fe \notin T)} = \frac{ {\bf P}(T \in \cA) - {\bf P}(T \in \cA, \fe\in T) }{1-{\bf P}(\fe \in T)}.$$
By multiplying denominators, we see that ${\bf P}(T \in \cA) \le {\bf P}(T \in \cA| \fe \notin T)$ if and only if
$$ {\bf P}(T \in \cA) - {\bf P}(T \in \cA, \fe\in T) \ge {\bf P}(T \in \cA) - {\bf P}(T \in \cA){\bf P}(\fe \in T)$$
which simplifies to ${\bf P}(T \in \cA) \ge {\bf P}(T \in \cA| \fe \in T)$.
\end{proof}

\begin{proof}[Proof of Lemma \ref{lem:WSF2}]

For $n\ge 0$, let $B(n)$ denote the ball of radius $n$ centered at the identity element in $C(\Gamma,f)$. For each $n$, choose a non-negative integer $i_n$ so that the following hold.
\begin{enumerate}
\item $\lim_{n\to\infty} i_n \to \infty$.
\item The quotient map $\pi_n$ restricted to $B(i_n)$ is injective but not surjective. Moreover, if $v,w$ are vertices in $B(i_n)$ then the number of edges in $C^f_n$  from $v\Gamma_n$ to $w\Gamma_n$ equals the number of edges in $B(i_n)$ from $v$ to $w$.
\end{enumerate}
Because $\bigcap_{n=1}^\infty \bigcup_{i \ge n} \Gamma_i = \{e\}$, it is possible to find such a sequence.

Let  $C^w_n$ denote the graph $C_n^f$ with all the vertices outside of $B(i_n)\Gamma_n$ contracted together. To be precise, $C^w_n$ has vertex set $B(i_n)\Gamma_n \cup \{*\}$. Every edge in $C_n^f$ with endpoints in $B(i_n)\Gamma_n$ is also in $C^w_n$. For every edge in $C_n^f$ with one endpoint $v$ in $B(i_n)\Gamma_n$ and the other endpoint not in $B(i_n)\Gamma_n$, there is an edge in $C^w_n$ from $v$ to $*$.

Similarly, let $D^w_n$ denote the graph $C(\Gamma,S)$ with all the vertices outside of $B(i_n)$ contracted together. By the choice of $i_n$, $D^w_n$ is isomorphic to $C^w_n$. Let $\nu^{C,w}_n$ be the law of the uniform spanning tree on $C^w_n$, $\nu^{D,w}_n$ be the law of the uniform spanning tree on $D^w_n$ and $\nu_n$ be the law of the uniform spanning tree on $C_n^f$.

Let $\cA \subset 2^E$ be an increasing set which depends on only a finite number of edges (i.e., there is a finite subset $F \subset E$ such that if $x, y \in 2^E$ and $x \cap F = y \cap F$ then $x \in \cA \Leftrightarrow y \in \cA$). If $n$ is sufficiently large, then $F \subset B(i_n)$. So we define $\cA_n \subset 2^{E_n}$ by: $x\in \cA_n \Leftrightarrow \exists y \in \cA$ such that $\pi_n(y \cap F) = x \cap \pi_n(F)$. By abuse of notation, we also consider $\cA_n$ to be a subset of the set of edges of $C^w_n$.

Because $C^w_n$ is obtained from $C_n^f$ by adding some edges and contracting some edges, repeated applications of the previous lemma imply $\nu^{C,w}_n(\cA_n) \le \nu_n(\cA_n)$. By definition, $\nu^{C,w}_n(\cA_n) = \nu^{D,w}_n(\cA)$ and $\nu_n(\cA_n) = \tilde{\nu}_n(\cA)$. Thus,
$$ \nu^{D,w}_n(\cA) \le \tilde{\nu}_n(\cA).$$

Let $E(i_n)$ denote the set of edges in the ball $B(i_n)$. We consider $2^{E(i_n)}$, the set of all subsets of $E(i_n)$, to be included in $2^E$, the set of all subsets of $E$, in the obvious way. By definition of the Wired Spanning Forest, the projection of $\nu^{D,w}_n$ to $2^{E(i_n)} \subset 2^E$ converges to $\nu_{WSF}$ in the weak* topology. So if $\tilde{\nu}_\infty$ is a weak* limit point of $\{\tilde{\nu}_n\}_{n=1}^\infty$ then we have
$$\nu_{WSF}( \cA) \le \tilde{\nu}_\infty( \cA)$$
for every increasing $\cA \subset 2^E$ which depends on only a finite number of edges. This means that, for any finite subset $F \subset E$, the projection of $\nu_{WSF}$ to $2^F$, denoted $\nu_{WSF}|2^F$, is stochastically dominated by $\tilde{\nu}_\infty|2^F$. By Strassen's theorem \cite{St65}, there exists a probability measure $J_F$ on
$$\{ (x,y) \in 2^F \times 2^F :~ x \subset y\}$$
with marginals $\nu_{WSF}|{2^F}$ and $\tilde{\nu}_\infty|{2^F}$ respectively. By taking a weak* limit point of $\{J_F\}_{F \subset E}$ as $F$ increases to $E$, we obtain the existence of a Borel probability measure $J$ on
$$\{ (x,y) \in 2^E \times 2^E :~ x \subset y\}$$
with marginals $\nu_{WSF}$ and $\tilde{\nu}_\infty$ respectively.

Observe that the average degree of a vertex in the WSF is 2. To put it more formally, for every $g\in \Gamma$, let $\deg_g:2^E \to \Z$ be the map $\deg_g(x)$ equals the number of edges in $x$ adjacent to $g$. By \cite[Theorem 6.4]{BLPS01}, $\int \deg_g(x)~d\nu_{WSF}(x)=2$. Also $\int \deg_g(x)~d\tilde{\nu}_\infty(x)=2$. This can be seen as follows. Because $\tilde{\nu}_n$ is $\Gamma$-invariant, it follows that $\int \deg_g(x)~d\tilde{\nu}_{n}(x)$ is just the average degree of a vertex in a uniformly random spanning tree of $C_n^f$. However, each such tree has $[\Gamma:\Gamma_n]$ vertices and $[\Gamma:\Gamma_n]-1$ edges and therefore, the average degree is $2([\Gamma:\Gamma_n]-1)[\Gamma:\Gamma_n]^{-1}$ which converges to $2$ as $n\to\infty$.

Because $\int \deg_g(x)~d\nu_{WSF}(x)=\int \deg_g(x)~d\tilde{\nu}_\infty(x)$ for every $g\in \Gamma$, it follows that $J$ is supported on $\{(x,x):~x \in 2^E\}$. Thus $\tilde{\nu}_\infty = \nu_{WSF}$ as claimed.
\end{proof}

For $x \in 2^E$, let $x_1$ denote the restriction of $x$ to the set of all edges containing the identity element. Let $\rho$ be the continuous pseudo-metric on $2^E$ defined by $\rho(x,y) = 1$ if $x_1\ne y_1$ and $\rho(x,y)=0$ otherwise. This pseudo-metric is dynamically generating.

For $x\in 2^{E_n}$, let $\phi_x: \Gamma/\Gamma_n \to 2^E$ be the map $\phi_x(g\Gamma_n) = \widetilde{gx}$. Let $W \subset \Gamma, L \subset C(2^E)$ be non-empty finite sets and $\delta>0$. Let $\BAD(W,L,\delta,\Gamma_n)$ be the set of all $x\in 2^{E_n}$ such that $\phi_x \notin \Map_{\nu_{WSF}}(W,L,\delta,\Gamma_n)$.

\begin{lemma}\label{lem:BAD2}
$$\lim_{n\to\infty} \nu_n(\BAD(W,L,\delta,\Gamma_n)) = 0.$$
\end{lemma}
\begin{proof}
The proof is similar to the proof of Lemma \ref{lem:BAD1}. It uses Lemma \ref{lem:WSF2} above and the fact that $\Gamma \cc (2^E,\nu_{WSF})$ is ergodic by \cite[Corollary 8.2]{BLPS01}.
\end{proof}

\begin{lemma}\label{lem:lower2}
$h_{\Sigma,\nu_{WSF}}(2^E,\Gamma) \ge \limsup_{n\to \infty} [\Gamma:\Gamma_n]^{-1} \log \tau(C_n^f).$
\end{lemma}

\begin{proof}
Let $W \subset \Gamma, L \subset C(2^E)$ be non-empty finite sets and $\delta>0$. Denote by $Y_n$ the set of spanning trees in $C_n^f$ not contained in
$\BAD(W,L,\delta,\Gamma_n)$.
By the previous lemma, $\lim_{n\to \infty} |Y_n|^{-1} \tau(G_n) = 1$.
By definition of $\rho_\infty$, if $x\ne y \in Y_n$ then $\rho_\infty(\phi_x,\phi_y)= 1$. Therefore, if $0<\epsilon<1$ then $\{\phi_y:~y\in Y_n\}$ is $\epsilon$-separated with respect to $\rho_\infty$ which implies
$$ N_\epsilon( \Map_{\nu_{WSF}}(W,L,\delta,\Gamma_n), \rho_\infty) \ge |Y_n|.$$
Because $\lim_{n\to \infty} |Y_n|^{-1}\tau(C_n^f) = 1$, this implies
$$\limsup_{n\to \infty} [\Gamma:\Gamma_n]^{-1} \log \tau(C_n^f) \le h_{\Sigma,\nu_{WSF}}(2^E,\Gamma).$$
\end{proof}

\subsection{A topological model}


Recall the definition of $\cF=\cF_f$ from the introduction. It is a closed $\Gamma$-invariant subset of $S_*^\Gamma$.
We refer the reader to \cite[Chapter I.8]{BH} for background about the end space of a topological space.

\begin{lemma}\label{lem:model}
If $\Gamma$ is not virtually $\Zb$ then $h_{\Sigma,\nu_{WSF}}(2^E,\Gamma) \le h_\Sigma(\cF,\Gamma)$.
\end{lemma}

\begin{proof}
Because $\Gamma$ is not virtually cyclic,  \cite[Theorem 10.1]{BLPS01} implies that for $\nu_{WSF}$-a.e. $x \in 2^E$, every component of $x$ is a 1-ended tree. Therefore, given such an $x$, for every $g\in \Gamma$ there is a unique edge $s_*\in S_*$ such that $gs_* \in x$ and if $C(x,g)$ is the connected component of $x$ containing $g$ then $C(x,g) \setminus \{gs_*\}$ has two components: a finite one containing $g$ and an infinite one containing $gs \in \Gamma$ (where $s=p(s_*) \in S$). Informally, $gs \in \Gamma$ is closer to the point at infinity of $C(x,g)$ than $g$ is. Let $\Phi(x)\in S_*^\Gamma$ be defined by $\Phi(x)_g=s_*$. Also let $\nu_{\cF} = (\Phi_*)\nu_{WSF}$. Because $\nu_{WSF}$-a.e. $x\in 2^E$ is such that every component of $x$ is a $1$-ended tree, it follows that $\nu_\cF$ is supported on $\cF$. The random oriented subgraph with law $\nu_\cF$ is called the {\em Oriented Wired Spanning Forest} (OWSF) in \cite{BLPS01}.

Note that $\Phi$ induces a measure-conjugacy from the action $\Gamma \cc (2^E, \nu_{WSF})$ to the action $\Gamma \cc (\cF,\nu_\cF)$. Thus $h_{\Sigma,\nu_{WSF}}(2^E,\Gamma) =h_{\Sigma,\nu_{\cF}}(\cF,\Gamma)$. Theorem \ref{thm:entropy} now  implies the lemma.
\end{proof}

\begin{lemma}\label{lem:model2}
If $\Gamma$ is virtually $\Zb$ then $h_{\Sigma,\nu_{WSF}}(2^E,\Gamma) \le h_\Sigma(\cF,\Gamma)$.
\end{lemma}

\begin{proof}
Let $\Ends(\Gamma)$ denote the space of ends of $C(\Gamma,f)$. Because $\Gamma$ is 2-ended, $|\Ends(\Gamma)|=2$. Let $x\in 2^E$ be connected and denote by $\Ends(x)$ the space of ends of $x$. Endow each edge in $C(\Gamma, f)$ with length $1$. Note that for any $g\in \Gamma$ and $r>0$ there exists $r'>0$ such that
if $g'\in \Gamma$ has geodesic distance at least $r'$ from $g$ in $x$, then $g'$ has geodesic distance at least $r$ from $g$ in $C(\Gamma, f)$.
Thus the argument in the proof of \cite[Proposition I.8.29]{BH} shows that the inclusion map of $x$ into $C(\Gamma,f)$ induces a map $\phi_x:\Ends(x) \to \Ends(\Gamma)$. We claim that this is a surjection.

Let $K \subset \Gamma$ be a finite set such that $C(\Gamma,f)\setminus K$ has two infinite components $\cC_0,\cC_1$ corresponding to the two ends $\eta_0,\eta_1$ of $C(\Gamma,f)$.

For each $i=0, 1$, define a subgraph $x|\cC_i$ of $C(\Gamma, f)$ as follows: it has the same vertices as $\cC_i$ does, and an edge $e$ in $E$ lies in $x|\cC_i$ exactly when $e$ is in both $x$ and $\cC_i$.
Because $x$ is connected, each component of $x|{\cC_i}$ contains an element of $KS$. Since $KS$ is finite, this implies that at least one of the components of $x|\cC_i$ must be infinite. Then any proper ray in this infinite component of $x|\cC_i$ gives rise to an end $\omega_i$ of $x$,
and also gives rise to an end of $C(\Gamma, f)$, which must be $\eta_i$.
It follows that $\phi_x(\omega_i)=\eta_i$. Because $i$ is arbitrary, $\phi_x$ is surjective as claimed.


By the claim, if $x \in 2^E$ is connected and 2-ended, we may identify $\Ends(x)$ with $\Ends(\Gamma)$ via the map $\phi_x$.

Given $(x,\eta) \in 2^E \times \Ends(\Gamma)$ with the property that $x$ is a 2-ended tree, we define $\Phi(x,\eta) \in \cF$ as follows.
For each $g \in \Gamma$, let $s_*\in S_*$ be the unique edge so that $x \setminus \{gs_*\}$
has two components: one containing $g$ (which is either finite or infinite with an end not equal to $\eta$),
the other containing $gp(s_*)$ and having an end equal to $\eta$.
Informally, $gp(s_*) \in \Gamma$ is ``closer'' to $\eta$ than $g$ is.
Let $\Phi(x,\eta)\in S_*^\Gamma$ be defined by $\Phi(x,\eta)_g=s_*$. Clearly $\Phi(x, \eta)\in \cF$.

Let $\zeta$ be the uniform probability measure on $\Ends(\Gamma)$. By \cite[Theorems 10.1 and 10.4]{BLPS01}, the WSF on $C(\Gamma,f)$ is a.s. a 2-ended tree. So  $\nu_{\cF} = (\Phi_*)(\nu_{WSF} \times \zeta)$ is well-defined.


The action of $\Gamma$ on $C(\Gamma,f)$ naturally extends to $\Ends(\Gamma)$. Note that $\Phi$ induces a measure-conjugacy from the action $\Gamma \cc (2^E \times \Ends(\Gamma), \nu_{WSF} \times \zeta)$ to the action $\Gamma \cc (\cF,\nu_\cF)$. By Theorem \ref{thm:entropy},
$$h_{\Sigma,\nu_{WSF} \times \zeta}(2^E \times \Ends(\Gamma),\Gamma) =h_{\Sigma,\nu_{\cF}}(\cF,\Gamma) \le h_\Sigma(\cF,\Gamma).$$

Because $\Gamma$ is virtually $\Z$, it is amenable \cite[Theorem G.2.1 and Proposition G.2.2]{BHV}. Thus by \cite[Theorem 1.2]{Bo12}  \cite[Theorem 6.7]{KL11b}, $h_{\Sigma,\nu_{WSF} \times \zeta}(2^E \times \Ends(\Gamma),\Gamma)$ is the classical entropy of the action, denoted by $h_{\nu_{WSF} \times \zeta}(2^E \times \Ends(\Gamma),\Gamma)$. It is well-known that classical entropy is additive under direct products. Thus,
\begin{eqnarray*}
h_{\nu_{WSF} \times \zeta}(2^E \times \Ends(\Gamma),\Gamma) &=& h_{\nu_{WSF}}(2^E,\Gamma) + h_{\zeta}(\Ends(\Gamma),\Gamma)= h_{\nu_{WSF}}(2^E,\Gamma).
\end{eqnarray*}
The last equality holds because $\Ends(\Gamma)$ is a finite set. By \cite[Theorem 1.2]{Bo12}  \cite[Theorem 6.7]{KL11b} again, $h_{\nu_{WSF}}(2^E,\Gamma) = h_{\Sigma,\nu_{WSF}}(2^E,\Gamma)$. Thus
$$h_{\Sigma,\nu_{WSF}}(2^E,\Gamma)=h_{\Sigma,\nu_{WSF} \times \zeta}(2^E \times \Ends(\Gamma),\Gamma) \le h_\Sigma(\cF,\Gamma).$$
\end{proof}

\subsection{The upper bound}




Given $s_* \in S_*$, let $\vs_*$ denote the {\em oriented} edge from $e$ to $p(s_*)$. For $x,y \in \cF$, let $\rho^\cF(x,y) = 1$ if $x_e \ne y_e$. Let $\rho^\cF(x,y)=0$ otherwise. This is a dynamically generating continuous pseudo-metric on $\cF$.

For $\phi:\Gamma/\Gamma_n \to \cF$, let $E(\phi)$ be the set of all edges of $C_n^f$ of the form $\pi_n(gs_*)$ where $\phi(g^{-1}\Gamma_n)_e=s_*$. Let $\BAD(\phi)\subset E(\phi)$ be those edges $\pi_n(gs_*)$ where $\phi(g^{-1}\Gamma_n)_e=s_*$ and $\phi( (gs)^{-1}\Gamma_n)_e = s_*^{-1}$ (where $s=p(s_*)$). Let $\GOOD(\phi) = E(\phi) \setminus \BAD(\phi)$. Also let $\vec{\GOOD}(\phi)$ be the set of all {\em oriented edges} of the form $\pi_n(g\vs_*)$ where $\phi(g^{-1}\Gamma_n)_e=s_*$ and $\phi( (gs)^{-1}\Gamma_n)_e \neq s^{-1}_*$ (where $s=p(s_*)$, so the corresponding unoriented edge is in $\GOOD(\phi)$). By abuse of notation, we will sometimes think of $\GOOD(\phi)$ and $\BAD(\phi)$ as subgraphs of $C_n^f$ but not in the usual way. To be precise, we consider $\GOOD(\phi)$ to be the smallest subgraph containing all the edges in $\GOOD(\phi)$ (and similarly, for $\BAD(\phi)$). Thus $\GOOD(\phi)$ and $\BAD(\phi)$ have no isolated vertices and, in general, are not spanning.

Because $\bigcap_{n=1}^\infty \bigcup_{i\ge n} \Gamma_i = \{e\}$, we may assume, without loss of generality, that for any $s_1\ne s_2 \in S \cup \{e\}$, $s_1\Gamma_n \ne s_2\Gamma_n$. An {\em oriented cycle} in $C_n^f$ is a sequence $g_0\Gamma_n,g_1\Gamma_n,\ldots, g_m \Gamma_n \in \Gamma/\Gamma_n$ such that $g_0\Gamma_n=g_m\Gamma_n$ and there exist $s_i \in S$ such that $\pi_n(g_is_i)= g_{i+1}\Gamma_n$. We consider two oriented cycles to be the same if they are equal up to a cyclic reordering of the vertices. Thus if $g_0\Gamma_n,g_1\Gamma_n,\ldots, g_m \Gamma_n$ is an oriented cycle then $g_i\Gamma_n, g_{i+1}\Gamma_n,\ldots, g_{m+i}\Gamma_n$ (indices $\mod m$) denotes the same oriented cycle. The cycle is {\em simple} if there does not exist $i,j$ with $0\le i < j<n$ such that $g_i\Gamma_n=g_j\Gamma_n$.  By definition, $\pi_n(g_0s_0s_1\cdots s_{m-1}) = g_0\Gamma_n$. The cycle is  {\em homotopically trivial} if $s_0s_1\cdots s_{m-1}$ is the identity element.

\begin{lemma}
Let $W \subset \Gamma$ be a symmetric finite set containing $S$ and  $\phi \in \Map(W,\delta,\Gamma_n)$ (where $\Map(W,\delta,\Gamma_n)$ is defined with respect to the pseudo-metric $\rho^\cF$ above). Then
\begin{enumerate}
\item $|\BAD(\phi)|   \le \delta^2 |S_*|^2 [\Gamma:\Gamma_n]$ and the number of vertices in $\BAD(\phi)$ is at most  $\delta^2|S| [\Gamma:\Gamma_n]$.
\item each component of $\GOOD(\phi)$ contains at most one cycle (i.e., each component is either a tree or is homotopic to a circle);
\item if a component $c$ of $\GOOD(\phi)$ is a tree, then there is a single vertex of $c$ incident with an edge in $\BAD(\phi)$;

\item for every integer $m>0$ there is an integer $N_m$ such that if $W \supset S^m\cup S$ and $n \ge N_m$ then the number of components of $\GOOD(\phi)$ is at most $(\delta^2|W| + m^{-1})[\Gamma:\Gamma_n]$;
\item if $n\ge N_m$ and $W \supset S^m\cup S$ then there is a spanning tree $T_\phi$ of $C_n^f$ such that $|T_\phi \Delta \GOOD(\phi)| \le (3\delta^2|W||S_*| + 2m^{-1})[\Gamma:\Gamma_n].$
\end{enumerate}
\end{lemma}

\begin{proof}
Let $\BAD(W,\phi)$ be the set of all vertices $g\Gamma_n \in \Gamma/\Gamma_n$ such that $\rho^\cF(w \circ \phi(g^{-1}\Gamma_n), \phi(wg^{-1}\Gamma_n)) =1$ for some $w\in W$. Because $\phi \in \Map(W,\delta,\Gamma_n)$, $|\BAD(W,\phi)| \le \delta^2 |W| [\Gamma:\Gamma_n]$.

We claim that the vertex set of $\BAD(\phi)$ is contained in $\BAD(W,\phi)$. So let $\pi_n(gs_*)  \in \BAD(\phi)$ for some $g\in \Gamma$, $s_* \in S_*$. Let $s=p(s_*)$. By definition, $\phi(g^{-1}\Gamma_n)_e=s_*$ and $\phi((gs)^{-1}\Gamma_n)_e=s^{-1}_*$. Because $\phi(g^{-1}\Gamma_n) \in \cF$, $\phi(g^{-1}\Gamma_n)_e=s_*$ implies
$$(s^{-1} \phi(g^{-1}\Gamma_n))_e=\phi(g^{-1}\Gamma_n)_s \ne s^{-1}_* =\phi(s^{-1}g^{-1}\Gamma_n)_e.$$
Thus $\rho^\cF(s^{-1} \phi(g^{-1}\Gamma_n), \phi(s^{-1}g^{-1}\Gamma_n))=1$ which implies $\pi_n(g) \in \BAD(W,\phi)$. By writing $\pi_n(gs_*)$ as $\pi_n(gs s_*^{-1})$ the same argument yields that $gs \in \BAD(W,\phi)$. So both endpoints of $\pi_n(gs_*)$ are in $\BAD(W,\phi)$. Because $\pi_n(gs_*) \in \BAD(\phi)$ is arbitrary, this implies the vertex set of $\BAD(\phi)$ is contained in $\BAD(W,\phi)$.

By choosing $W=S \subset \Gamma$ (by abuse of notation), we see that the number of vertices in $\BAD(\phi)$ is at most $\delta^2|S|[\Gamma:\Gamma_n]$. Since each vertex is incident to $|S_*|$ edges,  $|\BAD(\phi)|   \le \delta^2 |S_*|^2[\Gamma:\Gamma_n].$

Observe for every vertex $g\Gamma_n$ contained in $\GOOD(\phi)$ either
\begin{enumerate}
\item $g\Gamma_n$ is contained in both $\GOOD(\phi)$ and $\BAD(\phi)$ and there are no oriented edges of $\vec{\GOOD}(\phi)$ with tail $g\Gamma_n$, or
\item there is exactly one oriented edge $\pi_n(g\vs_*) \in \vec{\GOOD}(\phi)$ with tail $g\Gamma_n$.
\end{enumerate}

For every vertex $g_0\Gamma_n$ of $\GOOD(\phi)$, let $H(g_0\Gamma_n)$ be the set of all vertices ``ahead'' of $g_0\Gamma_n$. To be precise, this consists of all vertices $g_k\Gamma_n$ such that there exist oriented edges $\fe_0,\fe_1,\ldots, \fe_{k-1} \in \vec{\GOOD}(\phi)$ with $\fe_i=(g_i\Gamma_n,g_{i+1}\Gamma_n)$. Then two vertices $g\Gamma_n,g'\Gamma_n$ are in the same component of $\GOOD(\phi)$ if and only if $H(g\Gamma_n) \cap H(g'\Gamma_n) \ne \emptyset$ (one direction is obvious, the other  can be shown by induction on the distance between $g\Gamma_n$ and $g'\Gamma_n$ in the component of $\GOOD(\phi)$ containing both). Therefore, if $c$ is the collection of vertices in a connected component of $\GOOD(\phi)$ then
$$\bigcap_{g\Gamma_n \in c} H(g\Gamma_n)$$
is either a single vertex (contained in $\BAD(\phi)$) or a simple cycle. This implies items (2) and (3) in the statement of the lemma (since any cycle must be contained in $\bigcap_{g\Gamma_n \in c} H(g\Gamma_n)$).

Because $\bigcap_{n\in \N}\bigcup_{i\ge n}\Gamma_i=\{e\}$, there is an $N_m$ such that $n \ge N_m$ implies every homotopically nontrivial cycle in $C_n^f$ has length $> m$. Let us assume now that $W \supset S^m\cup S$  and $n\ge N_m$. We need to estimate the number of homotopically trivial oriented simple cycles of length $\le m$ in $\vec{\GOOD}(\phi)$.

So suppose that  $g_0\Gamma_n,g_1\Gamma_n,\ldots, g_k \Gamma_n = g_0\Gamma_n \in \Gamma/\Gamma_n$ is an oriented simple cycle in $\vec{\GOOD}(\phi)$ and $k \le m$. By definition, $(g_i\Gamma_n,g_{i+1}\Gamma_n) \in \vec{\GOOD}(\phi)$ for all $i$ (indices mod $k$). Thus if $s_{*,i} = \phi(g_i^{-1}\Gamma_n)_e$, $p(s_{*,i})=s_i$  and this cycle is homotopically trivial then $s_0s_1s_2\cdots s_{k-1}$ is the identity element.

By definition, $\cF$ does not contain any simple cycles. Therefore, there is some $i \le k-1$ such
$$((s_0\cdots s_i)^{-1}\phi(g_0^{-1}\Gamma_n))_e = \phi(g_0^{-1}\Gamma_n)_{s_0s_1\cdots s_i} \ne \phi( (g_0s_0s_1\cdots s_i)^{-1}\Gamma_n)_e.$$
Because $W \supset S^m$, $g_0\Gamma_n \in \BAD(W,\phi)$. Since $|\BAD(W,\phi)| \le \delta^2 |W| [\Gamma:\Gamma_n]$, this implies that the number of homotopically trivial simple cycles in $\GOOD(\phi)$ of length at most $m$ is at most $\delta^2 |W| [\Gamma:\Gamma_n]$.

Since each component of $\GOOD(\phi)$ either contains a vertex of  $\BAD(W,\phi)$ or contains a simple cycle of length $>m$, it follows that the number of components is at most $(\delta^2|W| + m^{-1})[\Gamma:\Gamma_n]$.

Let $L \subset \GOOD(\phi)$ be a set of edges such that each edge $\fe \in L$ is contained in a simple cycle of $\GOOD(\phi)$ and no two distinct edges $\fe_1, \fe_2 \in L$ are contained in the same simple cycle. So $|L| \le (\delta^2|W| + m^{-1})[\Gamma:\Gamma_n]$ and $\GOOD(\phi) \setminus L$ is a forest with at most $(\delta^2|W| + m^{-1})[\Gamma:\Gamma_n]$ connected components.

Note $\GOOD(\phi) \setminus L$ contains every vertex in $\GOOD(\phi)$ which contains at least $(1-\delta^2|S|)[\Gamma:\Gamma_n]$ vertices (because the number of vertices in $\BAD(\phi)$ is at most $\delta^2|S|[\Gamma:\Gamma_n]$).

An exercise shows that if $\cF_n \subset E_n$ is any forest contained in $C_n^f$ with at most $c(\cF_n)$ components and at least $v(\cF_n)$ vertices then there is a spanning tree $T_n \subset E_n$ such that $\cF_n \subset T_n$ and $|T_n \setminus \cF_n| \le c(\cF_n)+ [\Gamma:\Gamma_n] - v(\cF_n)$. In particular, this implies the last statement of the lemma.
\end{proof}

\begin{lemma}\label{lem:upper2}
$$h_\Sigma(\cF,\Gamma) \le \limsup_n [\Gamma:\Gamma_n]^{-1} \log \tau(C_n^f).$$
\end{lemma}

\begin{proof}
Let $1/2>\epsilon>0$. Let $m$ be a positive integer, $W \subset \Gamma$ be a finite set with $W \supset S^m\cup S$  and $\delta>0$. Let $n\ge N_m$ (where $N_m$ is as in the previous lemma). We assume $m$ is large enough  and $\delta$ is small enough so that $\epsilon > 6\delta^2|W||S_*|^2 + 4m^{-1}$. By Stirling's Formula, there is a constant $C>1$ such that for every $k$ with $0\le k \le \epsilon [\Gamma:\Gamma_n]$,
$${ [\Gamma:\Gamma_n] \choose k } \le C \exp(H(\epsilon ) [\Gamma:\Gamma_n]), \quad { [\Gamma:\Gamma_n]|S_*|/2 \choose k } \le C \exp(H(\epsilon ) [\Gamma:\Gamma_n]|S_*|/2),$$
where $H(\epsilon) := -\epsilon \log(\epsilon) - (1-\epsilon)\log(1-\epsilon)$.


If $\phi,\psi \in \Map(W,\delta,\Gamma_n)$ are such that $\vec{\GOOD}(\phi) = \vec{\GOOD}(\psi)$ and $\BAD(\phi)=\BAD(\psi)$, then $\rho^\cF(\phi(g\Gamma_n),\psi(g\Gamma_n))=0$ for all $g\in \Gamma$. On the other hand, $|\BAD(\phi)| \le  \delta^2 |S_*|^2 [\Gamma:\Gamma_n] \le \epsilon [\Gamma:\Gamma_n]$. Therefore,
\begin{eqnarray*}
&&N_0(\Map(W,\delta,\Gamma_n), \rho^\cF_2) \\
&\le& (\epsilon [\Gamma: \Gamma_n]+1) C \exp(H(\epsilon) [\Gamma:\Gamma_n]|S_*|/2) |\{\vec{\GOOD}(\phi):~ \phi \in \Map(W,\delta,\Gamma_n)\}|.
\end{eqnarray*}

Now suppose that $\phi, \psi \in \Map(W,\delta,\Gamma_n)$ are such that $\GOOD(\phi)=\GOOD(\psi)$ and $\Vert(\BAD(\phi))=\Vert(\BAD(\psi))$, where
 $\Vert(\BAD(\phi))$ and $\Vert(\BAD(\psi))$ denote the vertex sets of $\BAD(\phi)$ and $\BAD(\psi)$ respectively. Note that if $\fe \in \GOOD(\phi)$ is not contained in a simple cycle then the orientation of $\fe$ in $\vec{\GOOD}(\phi)$ is the same as the orientation of $\fe$ in $\vec{\GOOD}(\psi)$. This is because either $\fe$ is contained in a component which contains a simple cycle (in which case, $\fe$ must be oriented towards the simple cycle), or $\fe$ is contained in a component which contains a vertex of $\BAD(\phi)$ (in which case $\fe$ must be oriented towards that vertex).

On the other hand, for every simple cycle in $\GOOD(\phi)$, there are two possible orientations it can have in $\vec{\GOOD}(\psi)$. By the previous lemma, there are at most $(\delta^2|W| + m^{-1})[\Gamma:\Gamma_n] \le \epsilon[\Gamma:\Gamma_n]$ simple cycles. Therefore,
\begin{eqnarray*}
&&|\{{\vec \GOOD}(\phi):~ \phi \in \Map(W,\delta,\Gamma_n)\}| \\
&\le& 2^{\epsilon[\Gamma:\Gamma_n]} |\{(\GOOD(\phi), \Vert(\BAD(\phi))):~ \phi \in \Map(W,\delta,\Gamma_n)\}|.
\end{eqnarray*}
Because $|\Vert(\BAD(\phi))| \le \delta^2 |S|[\Gamma:\Gamma_n] \le \epsilon [\Gamma:\Gamma_n]$, it follows that,
\begin{eqnarray*}
&&|\{(\GOOD(\phi), \Vert(\BAD(\phi))):~ \phi \in \Map(W,\delta,\Gamma_n)\}| \\
&\le& (\epsilon[\Gamma:\Gamma_n]+1)C\exp( H(\epsilon)[\Gamma:\Gamma_n]) |\{ \GOOD(\phi) :~ \phi \in \Map(W,\delta,\Gamma_n)\}|.
\end{eqnarray*}

Now suppose that $\phi, \psi \in \Map(W,\delta,\Gamma_n)$ are such that $T_\phi = T_\psi$ (where $T_\phi, T_\psi$ is a choice of spanning tree as in the previous lemma). Then $|\GOOD(\phi) \Delta \GOOD(\psi)| \le  (6\delta^2|W||S_*| + 4m^{-1})[\Gamma:\Gamma_n] \le \epsilon [\Gamma:\Gamma_n]$. Therefore,
\begin{eqnarray*}
&&|\{ \GOOD(\phi) :~ \phi \in \Map(W,\delta,\Gamma_n)\}|\\
&\le& (\epsilon [\Gamma: \Gamma_n]+1) C\exp( H(\epsilon)[\Gamma:\Gamma_n] |S_*|/2) |\{ T_\phi :~ \phi \in \Map(W,\delta,\Gamma_n)\}|.
\end{eqnarray*}
We now have
$$N_0(\Map(W,\delta,\Gamma_n),\rho^\cF_2) \le (\epsilon [\Gamma: \Gamma_n]+1)^3 C^3 \exp\left( (\epsilon \log 2 + 2H(\epsilon) |S_*|) [\Gamma:\Gamma_n]\right) \tau(C_n^f).$$
Because $1/2>\epsilon>0$ is arbitrary and $H(\epsilon) \searrow 0$ as $\epsilon \searrow 0$, this implies the lemma.
\end{proof}

We are ready to prove Theorem~\ref{thm:WSF}.

\begin{proof}[Proof of Theorem \ref{thm:WSF}]
By Lemmas \ref{lem:lower2}, \ref{lem:model}, \ref{lem:model2} and  \ref{lem:upper2},
\begin{eqnarray*}
\limsup_{n\to \infty} [\Gamma:\Gamma_n]^{-1} \log \tau(C_n^f)  \le h_{\Sigma,\nu_{WSF}}(2^E,\Gamma) \le h_{\Sigma}(\cF,\Gamma)\le \limsup_{n\to \infty} [\Gamma:\Gamma_n]^{-1} \log \tau(C_n^f).
 \end{eqnarray*}
By Theorem 3.2 of \cite{Lyons05}, $  \lim_{n\to \infty} [\Gamma:\Gamma_n]^{-1} \log \tau(C_n^f) = {\bf h}(C(\Gamma,f))$. By Proposition~\ref{P-tree entropy}, $ {\bf h}(C(\Gamma,f)) = \log \det_{\cN\Gamma} f$ so this proves the theorem.
\end{proof}

\begin{question}
There is another natural topological model for uniform spanning forests. Namely, let $\cF_* \subset 2^E$ be the set of all subgraphs which are essential spanning forests. The word ``essential'' here means that every connected component of any $x\in \cF_*$ is infinite. What is the topological sofic entropy of $\Gamma \cc \cF_*$? Because  $\nu_{WSF}$ can naturally be realized as an invariant measure on $\cF_*$, the variational principle implies the topological sofic entropy of $\cF_*$ is at least the measure-theoretic sofic entropy of $\Gamma \cc (2^E, \nu_{WSF})$. 
\end{question}

\appendix

\section{Non-invertibility}

\begin{theorem}
Let $\Gamma$ be a countable group and $f\in \Rb \Gamma$ be well-balanced. Then $f$ is not invertible in the universal group $C^*$-algebra $C^*(\Gamma)$ (see \cite[Section 2.5]{BO}), or in $\ell^1(\Gamma)$. Moreover, $f$ is invertible in the group von Neumann algebra $\cN\Gamma$ if and only if $\Gamma$ is non-amenable.
\end{theorem}

\begin{proof}
We note first that $f$ is not invertible in
$C^*(\Gamma)$. Indeed, the trivial representation of $\Gamma$ gives rise to a unital $C^*$-algebra homomorphism $\varphi: C^*(\Gamma)\rightarrow \Cb$ sending every $s\in \Gamma$ to $1$. Since $\varphi(f)=0$ is not invertible in $\Cb$, $f$ is not invertible in $C^*(\Gamma)$.

Next we note that $f$ is not invertible in $\ell^1(\Gamma)$. This follows from the natural unital algebra embedding $\ell^1(\Gamma)\hookrightarrow C^*(\Gamma)$ being the identity map on $\Rb\Gamma$ and the non-invertibility of $f$ in $C^*(\Gamma)$.

Finally we note that $f$ is invertible in
$\cN\Gamma$ if and only if $\Gamma$ is non-amenable. Suppose that $\Gamma$ is amenable. Let $\{F_n\}_{n\in \Nb}$ be a right F{\o}lner sequence of $\Gamma$. Then $|F_n|^{-1/2}\chi_{F_n}$ is a unit vector in $\ell^2(\Gamma, \Cb)$ for each $n\in \Nb$, where $\chi_{F_n}$ denotes the characteristic function of $F_n$ in $\Gamma$,
and $\lim_{n\to \infty}\||F_n|^{-1/2}\chi_{F_n}\cdot f\|_2=0$. Thus $f$ is not invertible in $\cN\Gamma$. Now suppose that $\Gamma$ is non-amenable.
Set $\mu=-(f-f_e)/f_e\in \Rb\Gamma$ and denote by $\|\mu\|$ the operator norm of $\mu$ on $\ell^2(\Gamma, \Cb)$, which is also the norm of $\mu$ in $\cN\Gamma$.
Then $\mu$ is a symmetric probability measure on $\Gamma$ and the support of $\mu$ generates $\Gamma$.
Since $\Gamma$ is non-amenable,
Kesten's theorem \cite[Theorem 4.20.ii]{Paterson} implies that $\|\mu\|<1$. Therefore $f=f_e(1-\mu)$ is invertible in the Banach algebra $\cN\Gamma$.
\end{proof}



\begin{thebibliography}{999}
\Small

\bibitem{AL07} D. Aldous and R. Lyons. Processes on unimodular random networks.
{\it Electron. J. Probab.} {\bf 12} (2007), no. 54, 1454--1508.


\bibitem{BHV}
B. Bekka, P. de la Harpe, and A. Valette. {\it  Kazhdan's Property (T)}. New Mathematical Monographs, 11. Cambridge University Press, Cambridge, 2008.

\bibitem{BLPS01} I. Benjamini, R. Lyons, Y. Peres, and O. Schramm. Uniform spanning forests. {\it Ann. Probab.} {\bf 29} (2001), no. 1, 1--65.

\bibitem{Berg}
K. R. Berg. Convolution of invariant measures, maximal entropy.
{\it Math. Systems Theory}  {\bf 3}  (1969),  146--150.

\bibitem{Bo04}
L. Bowen. Couplings of uniform spanning forests. {\it Proc. Amer. Math. Soc.} {\bf 132} (2004), no. 7, 2151--2158.

\bibitem{Bo10a} L. Bowen. A measure-conjugacy invariant for actions of free groups. {\it Ann. of Math. (2)} {\bf 171} (2010), no. 2, 1387--1400.


\bibitem{Bo10} L. Bowen. Measure conjugacy invariants for actions of countable sofic groups.  {\it J. Amer. Math. Soc.} {\bf 23} (2010), no. 1, 217--245.

\bibitem{Bo11} L. Bowen. Entropy for expansive algebraic actions of residually finite groups. {\it Ergod.\ Th.\ Dynam. Sys.} {\bf 31} (2011), no. 3, 703--718.

\bibitem{Bo12} L. Bowen. Sofic entropy and amenable groups. {\it Ergod.\ Th.\ Dynam.\ Sys.} to appear.

\bibitem{BH}
M. R. Bridson and A. Haefliger.
{\it Metric Spaces of non-Positive Curvature}.
Grundlehren der Mathematischen Wissenschaften, 319. Springer-Verlag, Berlin, 1999.

\bibitem{BO}
N. P. Brown and N. Ozawa. {\it $C^*$-Algebras and Finite-Dimensional Approximations.}
Graduate Studies in Mathematics, 88. American Mathematical Society, Providence, RI, 2008.


\bibitem{BBI}
D. Burago, Y. Burago, and S. Ivanov. {\it  A Course in Metric Geometry}. Graduate Studies in Mathematics, 33.
American Mathematical Society, Providence, RI, 2001.


\bibitem{BP93} R. Burton and R. Pemantle. Local characteristics, entropy and limit theorems for spanning trees and
domino tilings via transfer-impedances. {\it Ann. Probab.} {\bf 21} (1993), no. 3, 1329--1371.


\bibitem{CL}
N. Chung and H. Li. Homoclinic group, IE group, and expansive algebraic actions. Preprint, 2011.


\bibitem{De06} C. Deninger. Fuglede-Kadison determinants and entropy for actions of discrete amenable
groups. {\it J. Amer. Math. Soc.} {\bf 19} (2006), no. 3, 737--758.

\bibitem{DS}
C. Deninger and K. Schmidt. Expansive algebraic actions of discrete
residually finite amenable groups and their entropy.
{\it Ergod.\ Th.\ Dynam.\ Sys.} {\bf 27} (2007), no. 3, 769--786.


\bibitem{ES05} G. Elek and E. Szab\'o. Hyperlinearity, essentially free actions and $L^2$-invariants. The sofic property.  {\it Math. Ann.}  {\bf 332}  (2005),  no. 2, 421--441.

\bibitem{ES06} G. Elek and E. Szab\'o. On sofic groups.  {\it J. Group Theory}  {\bf 9}  (2006),  no. 2, 161--171.


\bibitem{FK}
B. Fuglede and R. V. Kadison. Determinant theory in finite factors.
{\it Ann. of Math. (2)}  {\bf 55} (1952), 520--530.

\bibitem{FM92} Feder, T. and Mihail, M. (1992). Balanced matroids, Proc. 24th Annual ACM Sympos.
Theory Computing (Victoria, BC, Canada), pp. 26--38. ACM Press, New York.


\bibitem{Fu02}
A. Furman. Random walks on groups and random transformations. In: {\it Handbook of Dynamical
Systems, Vol. 1A}, North-Holland, Amsterdam, 2002, pp. 931--1014.

\bibitem{GR}
C. Godsil and G. Royle. {\it Algebraic Graph Theory}. Graduate Texts in Mathematics, 207. Springer-Verlag, New York, 2001.

\bibitem{Gr99} M. Gromov. Endomorphisms of symbolic algebraic varieties. {\it J. Eur. Math. Soc.} {\bf 1} (1999), no. 2, 109--197.


\bibitem{Kaminski}
B. Kami\'{n}ski. The theory of invariant partitions for $\Zb^d$-actions. Bull. Acad. Polon. Sci. S\'{e}r. Sci. Math. {\bf 29} (1981), no. 7-8, 349--362.

\bibitem{KL11a}
D. Kerr and H. Li. Entropy and the variational principle for actions of sofic groups. {\it Invent. Math.} to appear.

\bibitem{KL11b}
D. Kerr and H. Li. Soficity, amenability and dynamical entropy. {\it Amer. J. Math.} to appear.


\bibitem{Ki75} J. C. Kieffer. A generalized Shannon-McMillan theorem for the action of an amenable group on a probability space.
{\it Ann. Probab.} {\bf 3} (1975), no. 6, 1031--1037.

\bibitem{KS}
B. Kitchens and K. Schmidt.  Automorphisms of compact groups. {\it Ergod.\ Th.\ Dynam.\ Sys.}  {\bf 9} (1989), no. 4, 691--735.


\bibitem{Ko58} A. N. Kolmogorov. A new metric invariant of transient dynamical systems and automorphisms in Lebesgue spaces.  {\it Dokl. Akad. Nauk SSSR (N.S.)}  {\bf 119}  (1958), 861--864.

\bibitem{Ko59} A. N. Kolmogorov. Entropy per unit time as a metric invariant of automorphisms. {\it Dokl. Akad. Nauk SSSR}  {\bf 124}  (1959), 754--755.

\bibitem{Li}
H. Li. Compact group automorphisms, addition formulas and Fuglede-Kadison determinants. Preprint, 2010.

\bibitem{LS99}
D. Lind and K. Schmidt. Homoclinic points of algebraic $\Zb^d$-actions.
{\it J. Amer. Math. Soc.}  {\bf 12}  (1999),  no. 4, 953--980.



\bibitem{LS}
D. Lind and K. Schmidt. In preparation.

\bibitem{LSV}
D. Lind, K. Schmidt, and E. Verbitskiy.
Entropy and growth rate of periodic points of algebraic $\Zb^d$-actions.
In: {\it Dynamical Numbers: Interplay between Dynamical Systems and Number
  Theory}, ed. S. Kolyada, Yu. Manin, M. M\"{o}ller, P. Moree and T. Ward,
  Contemp. Math., vol. 532, American Mathematical Society, Providence,
  R.I., 2010.



\bibitem{LSW90} D. Lind, K. Schmidt and T. Ward. Mahler measure and entropy for commuting automorphisms of compact groups.  {\it Invent. Math.}  {\bf 101}  (1990),  no. 3, 593--629.





\bibitem{Luck94}
W. L\"{u}ck. Approximating $L^2$-invariants by their finite-dimensional analogues.
{\it Geom. Funct. Anal.}  {\bf 4}  (1994),  no. 4, 455--481.

\bibitem{Luck} W. L\"uck. {\it $L^2$-Invariants: Theory and Applications to Geometry and K-theory}. Springer-Verlag, Berlin, 2002.

\bibitem{Lyons05} R. Lyons. Asymptotic enumeration of spanning trees.  {\it Combin. Probab. Comput.} {\bf 14} (2005), no. 4, 491--522.

\bibitem{Lyons10}
R. Lyons. Identities and inequalities for tree entropy. {\it Combin. Probab. Comput.} {\bf 19} (2010), no. 2, 303--313.


\bibitem{Ol85} J. Moulin Ollagnier. \textit{Ergodic Theory and Statistical Mechanics}. Lecture notes in Math., 1115.
Springer, Berlin, 1985.


\bibitem{OW87} D. Ornstein and B. Weiss. Entropy and isomorphism theorems for actions of amenable groups.  {\it J. Analyse Math.}  {\bf 48}  (1987), 1--141.



\bibitem{Paterson}
A. L. T. Paterson. {\it Amenability.}
Mathematical Surveys and Monographs, 29.
American Mathematical Society, Providence, RI, 1988.

\bibitem{Pe91} R. Pemantle.  Choosing a spanning tree for the integer lattice uniformly. {\it Ann.
Probab.} {\bf 19} (1991), no. 4, 1559--1574.

\bibitem{Schmidt}
K. Schmidt. {\it Dynamical Systems of Algebraic Origin}.
Progress in Mathematics, 128. Birkh\"{a}user Verlag, Basel, 1995.

\bibitem{So98} R. Solomyak. On coincidence of entropies for two classes of dynamical systems. {\it Ergod.\ Th.\ Dynam.\ Sys.}
 {\bf 18} (1998), no. 3, 731--738.

\bibitem{St65} V. Strassen. The existence of probability measures with given marginals, Ann.
Math. Statist., 36, (1965), 423--439.

\bibitem{SV}
K. Schmidt and E. Verbitskiy. Abelian sandpiles and the harmonic model. {\it Comm. Math. Phys.} {\bf 292} (2009), no. 3, 721--759.

\bibitem{Varopoulos}
N. Th. Varopoulos. Long range estimates for Markov chains. {\it Bull. Sci. Math. (2)} {\bf 109} (1985), no. 3, 225--252.

\bibitem{We00} B. Weiss. Sofic groups and dynamical systems. Ergodic theory and Harmonic Analysis, Mumbai, 1999. {\it Sankhy\={a} Ser. A} {\bf 62}, (2000) no.3, 350--359.

\bibitem{Woess}
W. Woess. {\it Random Walks on Infinite Graphs and Groups}. Cambridge Tracts in Mathematics, 138. Cambridge University Press, Cambridge, 2000.

\bibitem{Yu65} S. A. Yuzvinskii. Metric properties of the endomorphisms of compact groups. (Russian)  {\it Izv. Akad. Nauk SSSR Ser. Mat.}  {\bf 29}  1965, 1295--1328. Translated in {\it Amer. Math. Soc. Transl. (2)} {\bf 66} (1968), 63--98.


\bibitem{Yu67} S. A. Yuzvinskii. Calculation of the entropy of a group-endomorphism. (Russian)  {\it Sibirsk. Mat. \^{Z}.}  {\bf 8}  (1967), 230--239.
Translated in {\it Siberian Math. J.} {\bf 8} (1967), 172--178.

\end{thebibliography}
\end{document}